\documentclass[12pt, twoside]{article}
\usepackage[a4paper,margin=2.4cm]{geometry}
\usepackage[english]{babel}
\usepackage{amsthm,hyperref,lmodern,mathrsfs,amssymb,amsmath,amsfonts}
\usepackage[utf8]{inputenc}
\usepackage{versions}
\usepackage{stmaryrd}
\usepackage{subcaption}
\usepackage{graphicx}
\usepackage{xspace}
\usepackage{color}
\usepackage{dsfont}
\usepackage{fontenc}

\usepackage{fancyhdr}
\usepackage{accents}
\usepackage{enumitem}

\allowdisplaybreaks

\numberwithin{equation}{section}

\newtheorem{theo}{Theorem}[section]
\newtheorem{lemme}[theo]{Lemma}
\newtheorem{prop}[theo]{Proposition}
\newtheorem{corol}[theo]{Corollary}
\newtheorem{defi}[theo]{Definition}
\newtheorem{remark}[theo]{Remark}
\newtheorem{example}[theo]{Example}

\newcommand{\ind}{\ensuremath{\mathds{1}}\xspace}

\newcommand{\E}{\ensuremath{\mathbb{E}}\xspace}

\newcommand{\N}{\ensuremath{\mathbb{N}}\xspace}

\newcommand{\R}{\ensuremath{\mathbb{R}}\xspace}

\newcommand{\CC}{\ensuremath{\mathcal{C}}\xspace}

\newcommand{\EE}{\ensuremath{\mathcal{E}}\xspace}
\newcommand{\FF}{\ensuremath{\mathcal{F}}\xspace}
\newcommand{\GG}{\ensuremath{\mathcal{G}}\xspace}
\newcommand{\HH}{\ensuremath{\mathcal{H}}\xspace}

\newcommand{\MM}{\ensuremath{\mathcal{M}}\xspace}
\newcommand{\NN}{\ensuremath{\mathcal{N}}\xspace}

\newcommand{\PP}{\ensuremath{\mathcal{P}}\xspace}

\renewcommand{\SS}{\ensuremath{\mathcal{S}}\xspace}
\newcommand{\TT}{\ensuremath{\mathcal{T}}\xspace}

\def\and{{\mathrm{{\rm and}}}}

\newcommand{\vf}[1]{\ensuremath{\stackrel{#1\rightharpoonup}{v}}\xspace}
\newcommand{\vb}[1]{\ensuremath{\stackrel{\leftharpoonup#1}{v}}\xspace}

\def\mathpal#1{\mathop{\mathchoice{\text{\rm #1}}%
   {\text{\rm #1}}{\text{\rm #1}}%
   {\text{\rm #1}}}\nolimits}

\def\Div{\mathpal{div}}

\def\id{\mathpal{id}}

\newenvironment{keywords}
 {\par \noindent\textit{Keywords:}\ \ignorespaces}
 {\par\vspace{3mm}} 

\title{\textsc{Brenier-Schrödinger problem on compact manifolds with boundary}}
\date{}

\author{\textsc{David Garc\'ia-Zelada and Baptiste Huguet}}

\begin{document}
\maketitle

%
%
%

%
%

\begin{abstract}\noindent
We consider
 the Brenier-Schrödinger problem on compact manifolds with boundary. In the spirit of
 a work by Arnaudon, Cruzeiro, Léonard and
 Zambrini, we study the kinetic property of 
 regular solutions and obtain 
 a link to the Navier-Stokes equations
with an impermeability condition. We also enhance the class of models for which the problem 
 admits a unique solution. 
This involves a method of taking quotients
by reflection groups for which we give
several examples.
\end{abstract}


\section{Introduction}\label{Section1}
\setcounter{equation}0

In mechanics, there are two classical
 dual descriptions of every phenomenon. 
 The first one is Newton's laws of motion (or
 Hamilton's equations). 
 They characterise the evolution 
 of a system by differential equations. The second one is the principle of least action. It characterises 
 the motion as the 
minimiser of a functional
constructed from
the kinetic and the potential energy. 
Applied to the evolution of perfect fluids, the first approach leads to the Euler equations
while the second approach sees
the evolution as a geodesic 
in the space of volume preserving diffeomorphisms
and it was developed by Arnold \cite{Arn}.
A relaxation of this problem
was proposed by 
Brenier \cite{Bre1}
where,
instead of seeking a flow, he looks for a measure
on the space of trajectories. His new problem is the minimisation of an average kinetic energy. The incompressibility constraint 
(i.e., volume preserving condition) 
 becomes a constraint on
 the marginals and the final endpoint condition becomes an endpoints measure constraint. 
 Brenier showed the accuracy of his problem 
by relating the solutions of the Euler equations 
with the solutions of his problem.

The problem treated in this article is the 
\emph{Brenier-Schrödinger problem}. This has been introduced in \cite{AAC} as a perturbation of Brenier's problem where the kinetic energy 
to be minimised is defined using 
a stochastic notion of velocity. 
This notion of velocity allows us
to think the problem as 
an entropy minimisation under marginal and endpoint constraints and allows us to use convex optimisation approaches.
It has been studied by several authors
  \cite{ACF,ACLZ,Bar2,Bar,BarL,BCN,Nen}.  In 
  the present article, we study this problem on compact manifolds with boundary and we work on 
the following two questions. 
\begin{itemize}[leftmargin=5mm]
\item The first one is the kinetics of 
the solutions. While 
Brenier's problem is linked to the Euler equations, 
the Brenier-Schrödinger problem is linked to the 
Navier-Stokes equations for which the viscosity term is a perturbation of the Euler equations. 
We prove that
the backward stochastic velocity of a regular
solution of the 
Brenier-Schrödinger problem is a solution of the Newtonian part of the Navier-Stokes equations.
This generalises to compact manifolds with
boundary the result of 
\cite{ACLZ} on the Euclidean space and the tori.
The main difference in our framework is the behavior of the solutions at the boundary. 
This can be found in Section
\ref{sec:ResultsKinematic}.
	\item The second question is about
the existence of a solution.  
We give a necessary and sufficient
condition for the existence of a unique
solution to the incompressible
Brenier-Schrödinger problem on homogeneous spaces.
This generalises the result given for the tori
in \cite{ACLZ}.
Moreover, we develop a method
to transport this result to
quotients by reflection groups.
Finally, we mention an additional 
example on the Euclidean space 
in a non-incompressible setting.
These results can be found in Section 
\ref{sec:ResultsExistence}.
\end{itemize}

Let us describe the model.
From now on, we fix a compact Riemannian manifold
$M$  with boundary $\partial M$,
interior $\mathring{M}$
and normalised Riemannian volume measure
$vol$, i.e., such that $vol(M) = 1$. 
In the present article,
we are interested 
in a minimisation problem
in the set $\mathcal P(\Omega)$
of probability measures
on the path space  
\[\Omega
= \{\omega \in M^{[0,1]}:
\, \omega 
\mbox{ is continuous}\},\]
which is endowed with the compact-open topology.
This minimisation problem
will be related to
the 
\emph{reflected Brownian motion}
on $M$ (restricted
to the time interval $[0,1]$).
This is the Markov process $\left(\beta_t\right)_{t \in [0,1]}$ 
whose generator is
the Laplacian on $M$ 
with a properly chosen domain. More precisely,
it satisfies that, for every 
$C^2$ function
$f: M \to \mathbb R$ such that
$\mathrm d f_x \cdot \nu_x = 0$
at every $x \in \partial M$, 
\[f(\beta_t) - \frac{1}{2} \int_0^t 
\Delta f (\beta_s) \mathrm d s\]
is a martingale
with respect to the filtration 
$\sigma\big((\beta_s)_{s \in [0,t]  }\big)$. 
See
also
\cite{AL}
 for an equivalent formulation,
 also recalled in
Section \ref{sec:Girsanov}.
Let us call $R \in \mathcal P(\Omega)$,
the law of 
the reflected Brownian motion
on $M$ whose initial position
follows the law $vol$
and $R^x \in \mathcal P(\Omega)$,
the law of 
the reflected Brownian motion
on $M$ whose
initial position is~${x \in \mathring M}$.

The object
to be minimised
is the so-called relative entropy
whose general definition is
the following.
For any measurable space $E$, 
the \textit{relative entropy}
of a probability
measure~${\mu\in \mathcal P(E)}$ 
with respect
to 
a probability measure $\nu \in \mathcal P(E)$
is 
\[H(\mu|\nu) = \int_E \rho \log \rho \,
\mathrm d \nu\]
if $\mathrm d \mu = \rho \, \mathrm d \nu$
and $H(\mu|\nu) = \infty$ if $\mu$ is not
absolutely continuous with
respect to $\nu$.

Now, 
let $\TT$ be a measurable subset of $[0,1]$, $(\mu_t)_{t\in\TT}$ a family of probability
measures on~$M$ indexed by $\TT$  and $\pi\in\PP(M\times M)$. 
We are interested in 
minimising $H(Q|R)$ among
all~$Q \in \mathcal P(\Omega)$
with the following constraints.
For any $t \in \mathcal T$, 
we ask that
$Q_t=\mu_t$, where
$Q_t$ is the image
measure (or pushforward) of $Q$
by the canonical map 
${X_t : \omega \in \Omega \mapsto 
\omega(t) \in M}$.
Additionally,
we ask that $Q_{01}=\pi$, where $Q_{01}$
is the
image measure of
$Q$ by the endpoints map~${(X_0,X_1) : \omega \in \Omega \mapsto 
\big(\omega(0),\omega(1)\big) 
\in M \times M}$.
This minimisation problem is 
 called the 
 \emph{Brenier-Schrödinger (or Brödinger, or Bredinger) problem}, it will be denoted by $\eqref{BS}$ in 
 this article
 and can be summarised as follows.
\begin{equation}\label{BS}
H(Q|R)\to \min,\, Q\in\PP(\Omega),\, [Q_t = \mu_t,\, \forall t\in\TT],\, Q_{01}=\pi.\tag{BS}
\end{equation}
It is a strictly convex problem with  convex constraints. 
Then,
the problem (BS) admits a unique solution
 if and only if there exists $Q\in\PP(\Omega)$ such that $Q_t=\mu_t$ for all $t\in\TT$,~$Q_{01}=\pi$ and $H(Q|R)<\infty$. 

A particular case is the \emph{incompressible
Brenier-Schrödinger problem}, denoted
by \eqref{iBS}. 
This is the case where
$\TT = [0,1]$, 
$\mu_t = vol$ for every $t \in [0,1]$ and $\pi$
has both marginals equal to $vol$.
\begin{equation}\label{iBS}
H(Q|R)\to \min,\, Q\in\PP(\Omega),\, [Q_t = vol,\, \forall t\in [0,1]],\, Q_{01}=\pi.\tag{iBS}
\end{equation}
We give some examples
where a solution to \eqref{iBS} exists
in Section \ref{sec:ResultsExistence}. 

A particular class
of solutions to \eqref{BS} are
introduced in \cite{ACLZ}.
By using the dual
maximisation problem, they have shown that
 if $P \in \mathcal P(\Omega)$
can be written as
\[\mathrm d P(X)=\exp\left(\eta(X_0,X_1) + \sum_{s\in\SS}\theta_s(X_s)+ \int_\TT p_t(X_t)\,
\mathrm dt\right) 
\mathrm dR(X)\]
for some bounded measurable
$\eta :M \times M\to\R$, $p:\TT\times M\to\R$ and $\theta :\SS \times M\to\R$, then~$P$
is a solution of \eqref{BS}
for the set $\TT \cup \SS$ and 
the marginals $\mu_t = P_t$. In fact,
\cite{ACLZ} describes the general form
of a solution by
relating the dual
maximisation problem and the primal 
minimisation problem. The existence of such solutions is proved in \cite{BarL} in the particular case of discrete problems, i.e for $\TT=\emptyset$.
In Section \ref{sec:ResultsKinematic}, we 
show that
 solutions of this form that are regular, 
 in a sense to be precised there, give rise
to solutions of the Navier-Stokes equations.

Let us summarise this article. In Section
\ref{sec:ResultsKinematic},
we present the result on the description of
regular solutions of \eqref{BS}
via the Navier-Stokes equations.
In Section \ref{sec:ResultsExistence}, we present the results on the existence of solution of \eqref{iBS} on compact manifolds and 
a non-incompressible version 
\eqref{E.gamma} on
$\mathbb R^n$ defined there.
In Section \ref{sec:Girsanov}, we develop the Girsanov theory to define the velocity of solutions. This makes a link between entropy minimisation and kinetic energy minimisation, see Remark 
\ref{rem:EntropyKineticEnergy}.
In Section \ref{sec:kin}, we give the proof of the kinetic results 
from Section \ref{sec:ResultsKinematic}.
Finally,
Section \ref{ex:sec.2}, 
Section \ref{ex:sec.3} and 
Section \ref{ex:sec.4} are dedicated to the proofs of 
the
results
about the existence from Section \ref{sec:ResultsExistence}.

\section{Results on the kinetics behavior}
\label{sec:ResultsKinematic}

Let $\TT$ be an
open subset of $[0,1]$ 
which is
a finite union of intervals
 and let $\SS$ be  a finite subset of $(0,1)$
such that $\TT \cap \SS=\emptyset$.
Following \cite{ACLZ}, we say that
 $P \in \mathcal P(\Omega)$
is a \emph{regular solution}
of \eqref{BS}
if it can be written as 
\begin{equation}
\label{eq:RegularSolutionForm}
\mathrm d P(X)=\exp\left(\eta(X_0,X_1) + \sum_{s\in\SS}\theta_s(X_s)+ \int_\TT p_r(X_r)\, \mathrm dr\right) 
\mathrm dR(X).
\end{equation}
for some 
regular enough functions
$\eta :M \times M\to\R$, $p:\TT\times M\to\R$ and $\theta :\SS \times M\to\R$. The regularity is such that
all equations in the theorem below makes sense 
($C^2$ would be enough, for instance,
but we are not interested
in attaining the least possible regularity).
As already explained,
from a dual-primal problem argument
explained in \cite{ACLZ}, a regular solution
is an actual solution of \eqref{BS}.

Now, let us introduce some 
further notation and comments before stating
the results.
\vspace{2mm}\\
\textbf{Forward regular solution.}
We will
say that a regular solution
$P$ is
a
\emph{forward regular solution}
of \eqref{BS} if, 
for every $x \in M$, the function
 $\psi^x: [0,1] \times M \to \mathbb R$ given by
\begin{equation}\label{eq:psi}
\psi^x_t(z) = \log\E_{R^x}\left[\left. \exp\left(\eta(x,X_1) + \sum_{s\in\SS 
\cap (t,1] }\theta_s(X_s) +\int_{\TT\cap (t,1]}p_r(X_r)\, \mathrm dr\right)
\right| X_t=z\right]
\end{equation}
is  $C^2$ in $z \in M$ and
$C^1$ in $t \in [0,1]\setminus \SS$ such that the function and its space first
derivatives are (jointly) càdlàg in $t$.
The times in 
$\SS$ are sometimes called \textit{shock times}.
These are the times where
$t \mapsto \psi^x_t(z)$ is discontinuous.
On the other hand,
the times in
$\TT$ are called
\textit{regular times}. 
The function $p$ can be thought of as a 
  \textit{pressure field} 
while the functions $\theta_s$ can be thought of as
  \textit{shock potentials}. 
\vspace{2mm}\\
\textbf{Backward regular solution.}
We will say
that a regular
solution~$P$ is a 
\emph{backward
regular solution}
of \eqref{BS} if
the time reversal of $P$
is a forward regular solution.
Equivalently, $P$ is a backward regular solution
of \eqref{BS} if, for
every $y \in M$, 
the function~${\varphi^y: [0,1] \times M \to \mathbb R}$ given by
\begin{equation}\label{eq:phi}
\varphi^y_t(z) =
 \log\E_{R^y}\left[\left. \exp\left(\eta(X_1,y) + \hspace{-5mm}
 \sum_{s\in\SS
 \cap [0,1-t)}
 \hspace{-3mm}
 \theta_s(X_{1-s}) +\int_{\TT\cap[0,1-t)}
 \hspace{-5mm}
 p_r(X_{1-r})\, \mathrm dr\right)
\right| X_t=z\right]
\end{equation}
is  $C^2$ in $z \in M$ and
$C^1$ in $t \in [0,1]\setminus \SS$
such that the function and its space first
derivatives are (jointly) càdlàg in $t$.
\vspace{2mm}\\
\textbf{Disintegration by the final position.}
We will be interested
in
the family $(\accentset{\leftharpoonup}{P}^y
)_{y \in M}$ 
of probability measures on $\Omega$ that satisfy
\begin{equation}\label{eq:Pbackward}
P=\int_M 
\accentset{\leftharpoonup}{P}^y \mathrm d P_1(y) 
\quad \mbox{ and }
\quad \accentset{\leftharpoonup}{P}^y
\left(\{\omega \in \Omega: 
\omega(1)=y\}\right) = 1
\mbox{ for every } y \in M.
\end{equation}
These can be thought of as the conditional laws
of $P$ given the final position
and are uniquely defined except for $x$
on a set
of $P_0$-measure zero.
\vspace{2mm}\\
\textbf{Logarithm.}
Suppose that $P$ has finite entropy
with respect to $R$. We will use the 
notion of \emph{stochastic velocity}
(or mean derivative) of $P$ introduced
initially by
Nelson  in \cite{Nel} for real processes. 
A presentation of the generalisation to manifold can be found in \cite{Gli}. 
Recall
that there exists
an open set $\mathcal N$ of 
$T\mathring{M}$
that contains the zero section
and
such that the
 exponential map~${\exp:\mathcal N \to \mathring{M}}$ 
is well-defined
and its restriction
$\exp_x:
\mathcal N \cap T_xM
\to \mathring{M}$
is a diffeomorphism onto
an open subset
$U(x) \subset M$.
Then, 
for $x \in M$
and $y \in U(x)$,
define
\[\overrightarrow{xy}=\log_x(y)\]
as the unique element of
$\mathcal N \cap T_xM$
such that 
$\exp(\overrightarrow{xy})
= y$.
Since $P$ is absolutely
continuous with respect to $R$,
we can see that,
for any $t$,
$X_t\notin \partial M$  for 
$P$-almost every $X$
since this also happens
for $R$-almost every $X$. 
The random times
\begin{equation}
\accentset{\leftharpoonup}{\tau_t} =
\frac{1}{2}
\inf 
\left\{h\leq 0 : X_{t-s}\in U(X_t)
\mbox{ for every }
s \in [h,0]
\right\}
\end{equation}
are strictly negative.
They allow us to define mean derivatives in a manifold.
\vspace{2mm}\\
\textbf{Covariant derivative and Laplacian.}
We denote
by $\nabla_u$ the
covariant derivative
in the direction of $u$
and by $\square$
the (negative definite) 
de Rham-Hodge-Laplace operator.
More precisely,
if 
$\delta$ the adjoint of the exterior differential from
the space of one-forms 
to the space of 
two-forms, then 
the de Rham-Hodge-Laplace operator
on one forms
is defined as~$-(\mathrm d\delta + \delta \mathrm d)$. 
The operator
$\square$ we need
is obtained
if we think
vector fields
as one-forms by using
the metric.
\vspace{2mm}\\
\textbf{Navier-Stokes equations.}
By the \emph{Navier-Stokes equations} on $M$, we refer to the differential system of unknown $(v,p)$,
together with an initial condition $v_0$,
\[
\left\{\begin{aligned}
&\left(\partial_t  +\nabla_{v_t}\right) v_t =\frac{1}{2}\square v_t-\nabla p_t, && 
t \in [0,1), \\
&\Div(v_t)=0, && t\in[0,1),\\
&\langle v(z),\nu_z\rangle = 0, && z\in\partial M,\\
\end{aligned}\right.
\]
where $\nu_z$ is the inward-pointing unit
normal to $z$.
Remark that our viscosity $\frac{1}{2}\square$ differs from \cite{ACF} where the Laplace operator on vector field is the divergence of the deformation tensor. However, in flat spaces, where Ricci curvature vanishes, both Laplace operators coincide. 
\vspace{2mm}\\
Now, we are ready to state one of our main results,
which is a generalisation of the Euclidean and tori results
\cite[Theorem 5.4]{ACLZ}.

\begin{theo}[Backward stochastic velocity and the Navier-Stokes
equations]
\label{theo:NavierStokes}

Suppose~$P$
is a backward regular solution
of \eqref{BS} with associated function $\varphi$.
 Then, 
for $P_1$-almost every
$y$, 
there exists
a measurable function called
the \emph{backward stochastic velocity}
\[\vb{y}_{\hspace{-1mm}}
\hspace{1mm}:[0,1]\times \Omega
\to TM\]
such that
$t \mapsto
\vb{y}_{\hspace{-1mm}t}
\hspace{-1mm}
(\omega)$ is 
left-continuous and has right
limits for every
$\omega \in \Omega$ and
such that
for every $t \in [0,1]$
we have
that, for
$P$-almost every $X$,
\begin{equation*}
\lim_{h\to 0^+}\frac{1}{h}\E_{\accentset{\leftharpoonup}{P}^y}
\left[\left. -
\overrightarrow{
X_{t}X_{t-h 
\wedge  \accentset{\leftharpoonup}{\tau}_t}}
\right|X_{[t,1]}\right]
=\vb{y}_{\hspace{-1mm}t}
\hspace{-1mm}(X).
\end{equation*}
Let $U_t^y(z) = - 
\nabla \varphi_{1-t}^y(z)$. 
Then, for $P_1$-almost every
$y \in M$,
 \[\vb{y}_t = U_t^y(X_t)
 \quad 
 P\text{-almost surely}.\]
Moreover,
for $P_1$ almost every $y\in M$,
the time-dependent
vector field
$U^y$ satisfies 
\begin{equation}
\label{eq:NavierStokes}
\left\{\begin{aligned}
&\left(\partial_t  +\nabla_{U_t^y}\right) U_t^y =\frac{1}{2}\square U_t^y-\ind_\TT(t)\nabla p_t, && 
t \in [0,1)\setminus 
\SS, \\
&U^y_{t^+}- U^y_{t} = 
\nabla \theta_t, && t\in S,\\
&\langle U^y(z),\nu_z\rangle = 0, && z\in\partial M,\\
&U^y_0=-\nabla\big(\eta(\cdot,y)\big), && 
t=0.\\
\end{aligned}\right.
\end{equation}

\end{theo}

The first equation in \eqref{eq:NavierStokes}
is the Newtonian part of the Navier-Stokes
equations 
while 
the second equation describes the evolution at 
the shock times. The third equation 
tells us 
the behavior at the boundary of the domain,
it says that the
stochastic velocity  
satisfies the impermeability condition. 
The fourth is the initial condition of the problem. 
Nevertheless, the velocity
does not seem to satisfy any continuity equation. 
The same approach on forward velocity results on a time reversed Navier-Stokes equation (or Navier-Stokes equation with negative viscosity)
which is stated in Corollary \ref{cor:ForwardNavierStokes}.

A continuity equation
is satisfied for
a combination of
averaged forward and
backward velocities,
$\stackrel{\hookrightarrow}{v}$ and $\stackrel{\hookleftarrow}{v}$.
These are defined by
\begin{equation}
\stackrel{\hookrightarrow}{v}_t(z) = 
\E_{P}\left[\left. 
\vf{X_0}_{\hspace{-2mm}t}
\hspace{2mm} \right|X_t = z\right] \quad\text{and}\quad\stackrel{\hookleftarrow}{v}_t(z) = 
\E_P\left[\left. 
\vb{X_1}_{\hspace{-2mm}t}
\hspace{2mm} \right|X_t = z\right],
\end{equation}
where $\hspace{-1mm}
\vf{x}_{\hspace{-1mm}t}
$ and 
$\hspace{-1mm} \vb{y}_{\hspace{-1mm}t}$
are defined in
Theorem 
\ref{theo:NavierStokes}
and Corollary 
\ref{cor:ForwardNavierStokes}. They are also shown to be measurable in $x$ and $y$ respectively
in these results.
The \emph{current velocity}
is defined as
\begin{equation}
v_{cu}
=\frac{1}{2}
\big( \hspace{-1mm}
\stackrel{\hookrightarrow}{v}_t+ \stackrel{\hookleftarrow}{v}_t \hspace{-1.2mm}
\big)
,\, \forall t\in[0,1]
\end{equation}
and it satisfies
the following continuity equation
which generalise 
\cite[Theorem 5.4]{ACLZ}.
\begin{theo}[Continuity 
equation]
\label{theo:ContinuityEquation}

Assume that (BS) admits a forward and backward regular solution $P$.
Then $v_{cu}$ satisfies
\begin{equation}
\partial_t \mu_t +\Div(\mu_t v_{cu}) = 0,\, \forall t\in\TT.
\end{equation}
\end{theo}
Actually, this result is not particular to bi-regular solutions, nor to solution. In fact, the proof uses only that $P$ is a semi-martingale measure with finite entropy with respect to the reflected Brownian motion. There, the average velocity shall be defined using the drift given by Girsanov theorem, more precisely, by applying Theorem \ref{theo:Girsanov} to $P$ and $R$ instead of $P^x$ and $R^x$. 
This equation has to be understood in distribution sense,
i.e., 
for all $f\in\CC^\infty(M)$ and every $t\in [0,1]$,
\begin{equation}
 \mu_t(f) + \int_0^t\mu_s(\langle 
 \mathrm df, v_{cu}\rangle)\mathrm ds = 0.
\end{equation}
In the 
incompressible case
\eqref{iBS}, where $\mu_t = vol$, the continuity equation becomes 
the incompressibility condition
$\Div(v_{cu}) = 0$.
Nevertheless, $v_{cu}$ does not satisfy the
Navier-Stokes equations.

\section{Results on the existence
of solutions}
\label{sec:ResultsExistence}

Let $M$ be a homogeneous compact Riemannian manifold,
i.e., one where the isometries act transitively.
We consider the incompressible Brenier-Schrödinger problem $\eqref{iBS}$. We may
write explicitly the dependence on $\pi$ and $M$
by $\eqref{iBS}_{M,\pi}$.

\begin{theo}[Existence for
homogeneous spaces]
\label{theo:Homogeneous}
The $\eqref{iBS}_{M,\pi}$ problem on a homogeneous
compact manifold $M$
admits a unique solution if and only if 
$H(\pi |vol\otimes vol)<\infty$.
\end{theo}

We will show that the property of
existence of solutions
is preserved under 
nice quotients.
Our setting will be the following.
Suppose that $M$ is  a connected
compact Riemannian manifold 
(without boundary)
and that
 $G$ is a finite group of isometries
of~$M$. 
For~${x \in M}$, consider
the stabiliser group
$G_x = \left\{ g  \in G:\, g(x) = x \right\}$,
and the induced subgroup~$\mathbb G_x=\{\mathrm d g_x \in O(T_x M):\,
g \in G_x\}$
of the orthogonal group of $T_x M$.
Let $R_x$ be the set of reflections in 
$\mathbb G_x$, i.e., $T \in \mathbb G_x$ 
belongs to $R_x$ if and only if
$\{u \in T_xM:\, Tu = u\}$  has
codimension one as a subspace of $T_xM$.

\begin{defi}[Reflection group]
We shall say that $G$
is a reflection group (of 
isometries) if  $\mathbb G_x$ is
the group generated by $R_x$
for every $x \in M$.
\end{defi}

We will be interested in the set
$N =M/G$
which has a topological structure
induced by the quotient map
$q:M \to N$.
Suppose that $G$ is a reflection group. We shall make of~$N$ a manifold with
corners. But first, let us recall 
the definition.

\begin{defi}[Manifold with
corners]

Let $N$ be a Hausdorff
second countable 
topological space and let 
$n > 0$ be a positive integer. 
Suppose that we have a family
$\{\varphi_\lambda \}_{\lambda \in \Lambda}$ of
homeomorphisms
\[\varphi_{\lambda}:
U_\lambda \subset N \to V_\lambda \subset 
[0,\infty)^n\]
where $U_\lambda$ 
(respectively $V_\lambda$)
is an open subset of
$N$ (respectively of $[0,\infty)^n$).
We say that the family 
$\left(\varphi_\lambda \right)_{\lambda \in \Lambda}$ 
is a \emph{smooth atlas
with corners}
if
\[\bigcup_{\lambda \in \Lambda} 
U_\lambda = N\]
and, for every $\mu,\nu \in \Lambda$,
\[\varphi_{\mu}^{\ }
\circ
\varphi_{\nu}^{-1}:
\varphi_\nu
\left(U_{\mu} \cap U_{\nu}\right)
\to
\varphi_\mu
\left(U_{\mu} \cap U_{\nu}\right)
\]
has a smooth extension to an open subset of $\R^n$.
We will refer to
$(N,\left(\varphi_\lambda \right)_{\lambda \in \Lambda})$ as a 
\emph{manifold
with corners}.

\end{defi}

It will be useful to have in
mind triangles (and
squares)  as the prototypical examples, the smooth atlas being
given by the set of all
diffeomorphisms
from an open subset
of the triangle (or square) to the open subsets
of $[0,\infty)^n$.
The notions of
tangent bundle, Riemannian metric, Levi-Civita connection 
and stochastic differential equations
can be carried
over to manifolds with corners.

\begin{lemme}[Quotient differentiable structure] \label{lem:RiemannianMwithCorners}
Suppose that $G$
is a reflection group of isometries
of $M$.
Then, 
\[N = M/G
 \mbox{ has a (unique) structure
of a Riemannian manifold with corners }\]
such that,
for every $x \in M$, there exists
a neighborhood $U \subset N$ 
of $q(x)$
together with
 an isometric immersion
$s: U \to M$
that is a local inverse of $q$, i.e.,
such that
\[q \circ s(x) = x \mbox{ for every }
x \in U.\]
\end{lemme}

We fix some notation about
the boundary of $N$.
The set of points that, by some 
chart~$\varphi_{\lambda}$, correspond
to points of the
(topological) boundary of 
$[0,\infty)^n \subset \mathbb R^n$
will be called the \emph{boundary}
and will be denoted by
$\partial N$.
The points that
correspond
to the singular points of
the boundary of 
$[0,\infty)^n$ will be called
the \emph{corner points} and the set
consisting of them will be denoted by
$\mathcal CN$. 
A boundary point $x$ that
is not a corner point will be called
a \emph{regular boundary point}
and there is a 
well-defined unit inward-pointing normal
vector~$\nu_x \in T_xN$ at $x$.
The complement of $\partial N$,
called the interior
of $N$, will be denoted by~$\mathring{N}$.

We will be interested on the 
reflected Brownian motion on these
manifolds. Similarly to the case of
manifolds with boundary, \emph{the reflected
Brownian motion} on $N$ is
a continuous
stochastic process $\left(\beta_t\right)_{t \in [0,1]}$
on $N \setminus \mathcal CN$
that satisfies the following condition. For every
$C^2$ function
$f: N \to \mathbb R$
such that
$\mathrm d f_x \cdot \nu_x = 0$
at every regular boundary point $x$,
we have that
\[f(X_t) - \int_0^t 
\Delta f (X_s) \mathrm d s\]
is a martingale
with respect to the filtration
$
\sigma(\left(\beta_s\right)_{s \in [0,t]})$.

\begin{theo}[Existence for quotients]
\label{th:ToQuotient}
Let $N$ be the quotient of $M$
by a reflection group and denote by
$q:M \to N$ its quotient map.
Let $\pi$ be a probability
measure on $M\times M$
with both marginals equal to $vol$
and such that
$H(\pi|vol \otimes vol ) < \infty$.
Then,
\begin{center}
\eqref{iBS}$_{M,\pi}$
admits a solution
$\Rightarrow$
\eqref{iBS}$_{N,(q\times q)_* \pi}$
admits a solution.
\end{center}
Moreover, if \eqref{iBS}$_{M,\pi}$ admits
a solution for every $\pi$ with finite entropy, then
\eqref{iBS}$_{N,\widetilde \pi}$ admits
a solution for every $\widetilde \pi$ with finite entropy.
\end{theo}

We end this section with a more exotic example.
 Let $\Omega$ the space of paths from~$[0,1]$ to~$M=\R^n$. 
Choose any probability measure $\chi 
\in \mathcal P(\mathbb R^n)$ and
let $R$ 
be
the Brownian motion whose initial position
has law $\chi$. 
We are looking to the following problem 
\begin{equation}\label{E.gamma}
H(P|R)\to \min; \big[P_t=\NN(0,1/4 \id), \forall 
t \in [0,1] \big], P_{01}=\pi,\tag{$\text{BS}_\gamma$}
\end{equation}
where $\pi$ is a probability measure 
on $M^2$. Through this example, we intend to challenge the assumptions of compactness and incompressibility.

\begin{theo}[Existence for Gaussian marginals]\label{theo:Gaussian}
The Brenier-Schrödinger problem \ref{E.gamma} admits a unique solution if and only if $H(\pi|R_{01})<\infty$.
\end{theo}

\section{Girsanov theorem}\label{sec:gene}
\label{sec:Girsanov}
\setcounter{equation}0

This section contains 
a version of Girsanov theorem
which is a 
translation to a manifold setting of the results from \cite{leo12b}.
This will make a link between entropy and kinetic energy and
will be useful for the
proof of Theorem \ref{theo:NavierStokes} and Theorem \ref{theo:ContinuityEquation}.

We need to use a 
different but equivalent
description
of the 
reflected Brownian motion.
Consider any vector field
 $\nu:M\to TM$ 
such that $\nu|_{\partial M}$ is 
the inward-pointing unit normal vector field. 
By using an embedding of $M$ into 
an Euclidean space, we may construct
a smooth family
$(\sigma_x)_{x \in M}$ of linear maps
$\sigma_x:\mathbb R^p \to T_x M$
such that $\sigma_x \sigma_x^* = \mathrm{id}_{T_xM}$.
By smooth we mean that
 the map $\sigma:M \times \mathbb R^p
\to TM$ defined by $\sigma(x,w) = \sigma_x(w)$
is smooth.
The reflected Brownian motion on $M$
can be 
defined as a semi-martingale
$(\beta_t)_{t \in [0,1]}$ on $M$  that
 solves the following Skorokhod problem.  
There exists a Brownian motion in $\R^p$ 
and a non-decreasing process 
$(L_s)_{s\in[0,1]}$ such that 
\[\mathrm d\beta_t =
\sigma(\beta_t) 
\mathrm dW_t + \nu_{\beta_t}\mathrm dL_t
\quad \mbox{ and }
\int_0^1{\ind_{\mathring{M}}(\beta_s)
\mathrm dL_s} = 0.\]
More information can be found
in \cite{AL}.
We may notice that the process $L$ is the local time of~$\beta$ at $\partial M$ and that $\beta$
is a reflected Brownian motion 
in the sense defined in the introduction, Section
\ref{Section1}.

Now, we are interested in the description of 
solutions whenever they exist. 
Recall that $R$ denotes
the law of $(\beta_t)
_{t \in [0,1]}$
whose initial position
follows the law $vol$.
Since the reference measure $R$ is a semi-martingale measure, the classical Girsanov theory implies 
that a solution $P$ of (BS) will also be a semi-martingale. Moreover, using the finite entropy condition, we have stronger boundedness properties on the Girsanov velocity vector field. 
Theorem \ref{theo:Girsanov} and Theorem 
\ref{theo:GirsanovDensity} 
below are adaptations
of results from \cite{leo12b}
to a manifold setting. They use 
a variational viewpoint 
of the entropy to improve Girsanov theorem under a finite entropy condition. We give here sketches of the proofs and objects in a manifold language. For a wider view, see \cite{leo12b} for the $\R^n$ setting and \cite{Hug2} for manifolds.
Let $B$ be the drift, defined on $1$-form valued processes by
\begin{equation}
\label{eq:DriftDefinition}
B_t(\alpha, \omega) = \int_0^t\langle\alpha_s(\omega_s), \nu_{\omega_s}\rangle\, \mathrm dL_s(\omega),
\end{equation}
where $\alpha_t\in\Gamma(T^*M)$ for all $t\in[0,1]$ and $\omega\in\Omega$. Notice that,
since $L$ is a function of bounded variation
uniquely defined except for a set of $R$-measure zero,
$B$ from 
\eqref{eq:DriftDefinition} is also 
uniquely defined except for a set of $R$-measure zero.
Let $A$ be the quadratic variation defined on bilinear form valued processes by 
\begin{equation}
\label{eq:QuadraticVariationDefinition}
A_t(h,\omega)= \int_0^t \left\langle 
\mathrm{Tr}\big(h_s(\omega_s)\big)
\right\rangle\, \mathrm ds,
\end{equation}
where $h_t\in\Gamma(T^*M\otimes T^*M)$ for all $t\in[0,1]$. 
Here $\langle \mathrm{Tr}(h)\rangle$
denotes the contraction of~$h$
using the Riemannian metric, so that
\eqref{eq:QuadraticVariationDefinition} is
well-defined everywhere and not just almost everywhere.
With this notation, the measure $R$ satisfies the martingale problem
$\MM\PP(B,A)$, i.e., $R$ is the unique 
probability measure such that
\begin{equation*}
M^f_t := f(X_t) - f(X_0) - B_t(\mathrm df) -
\frac{1}{2}A_t(\mathrm{Hess} f),
\end{equation*}
is an $R$-local martingale
for every $f\in\CC^\infty(M)$. 
We denote by 
$\mathrm dX$ its Itô derivative and by 
$\mathrm d^R_mX$ the martingale part of $\mathrm dX$ with respect to $R$ (see \cite[Definition 7.33]{Eme}). Both are infinitesimal vector fields.
The problem $\MM\PP(B,A)$ implies
 that
\begin{equation*}
\mathrm dX_t = 
\mathrm d^R_mX_t + \mathrm dB_t,\, R\text{-almost surely},
\end{equation*}
and 
\begin{equation*}
\mathrm d[X,X]_t = \mathrm dA_t,\, R\text{-almost surely}.
\end{equation*}
For the version
of Girsanov theorem we are interested
in, we will use the space
 $\GG$ of measurable functions $g : [0,1]\times \Omega \to T^*M$ such that $g_t(\omega)\in T^*_{\omega_t}M$
for every $t\in[0,1]$ and
for every $\omega\in\Omega$. 
For any probability measure $Q$ on $\Omega$, 
we define the semi-norm on $\GG$ 
\begin{equation*}
\|g\|_{Q}=\E_Q\left[\int_0^1 
\|g_t\|_{T^* M}^2 \mathrm d t\right]^{1/2}\hspace{-2mm}
= \E_Q\left[A_1(g\otimes g)\right]^{1/2}.
\end{equation*}
Identifying functions by using the semi-norm $\|.\|_{Q}$, we 
define the Hilbert spaces 
\begin{equation*}
\GG(Q)=\left\{g\in\GG : \|g\|_{Q}
<+\infty\right\}\quad\text{and}\quad \HH(Q) = \{g\in\GG(Q) : g\, \text{adapted}\}.
\end{equation*} 
Adapted means here that,
for every $t \in [0,1]$, the map
$\omega \mapsto g(t,\omega)$ is 
measurable with respect to the
completion of $\sigma((X_s)_{s \in [0,t]})$
using the Brownian motion law $R$,
where the map $X_s:\Omega \to M$
is the projection map
$X_s(\omega) = \omega_s$.

The following result is Girsanov theorem
for the family $(P^x)_{x \in M}$
of probability measures on $\Omega$ that satisfy
\[P=\int_M P^x \mathrm d P_0(x) 
\quad \mbox{ and }
\quad P^x\left(\{\omega \in \Omega: 
\omega(0)=x\}\right) = 1
\mbox{ for every } x \in M.\]
By taking a time reversal, it would
tell us something about
the family
$(\accentset{\leftharpoonup}{P}^y
)_{y \in M}$, defined in~\eqref{eq:Pbackward}, but we will not use this until
later.

\begin{theo}[Girsanov theorem]
\label{theo:Girsanov}
Let $P$ be such that $H(P|R)<\infty$.
Then,  for $P_0$-almost every $x \in M$,
the probability measure
$P^x$ 
is 
the law of a semi-martingale and there exists an adapted process $\zeta^x \in\HH(P^x)$ such that 
\[P^x\in\MM\PP(B+\hat{B}^x, A)
\quad \mbox{ and } \quad 
\hat{B}^x = 
 A(\zeta^x\otimes \cdot).\]

\end{theo}

\begin{remark}
In other words, 
$P^x$-almost surely,
 $\mathrm dX_t = 
\mathrm d_m^{P^x}X_t + \mathrm dA_t(\zeta^x\otimes\cdot) 
+ \nu_{X_t} \mathrm dL_t(X)$,  where the 
$P^x$-martingale part is equal 
to ~${\mathrm d_m^{P^x}X_t = \mathrm d_m^{R^x}X_t - 
\mathrm dA_t(\zeta^x\otimes\cdot)}$.
\end{remark}

\begin{proof}

Due to the chain rule for the entropy
\cite[Theorem C.3.1]{Dupuis}
\[H(P|R) = H(P_0|R_0)
+ \int_M H(P^x|R^x)
\mathrm d P_0(x),\]
we have that
$H(P^x|R^x)< \infty$
for $P_0$-almost
every $x \in M$.

For $h\in \HH(P^x)$, we would like
to define the processes $N^h$ by 
\begin{equation}
\label{eq:IntegralN}
N^h_t=\int_0^t\langle h_s, \mathrm d^{R^x}_m
X_s\rangle,\, 0\leq t\leq 1. 
\end{equation}
If $h\in\HH(P^x)\cap\HH(R^x)$, the process $N^h$ 
can be defined by \eqref{eq:IntegralN}. Its stochastic exponential~$\EE(N^h)= \exp(N^h - 
\frac{1}{2}[N^h,N^h])$ is a positive local martingale, so that it is a super-martingale and 
\begin{equation}\label{G.E5}
0\leq\E_{R^x}[\EE(N^h)_1]\leq 1 .
\end{equation} 
For $h\in\HH(P^x)\cap\HH(R^x)$, let $u$ be the function 
$u : \omega\in\Omega\mapsto 
N^h_1-\frac{1}{2}[N^h,N^h]_1$. The variational definition of the entropy, known as the
Donsker-Varadhan variational formula, implies that
\begin{equation*}
\E_{P^x}[u]-\log\E_{R^x}[e^u]\leq H(P^x|R^x)<+\infty.
\end{equation*}
Using \eqref{G.E5}, we have that
$$\E_{P^x}[u]\leq H(P^x|R^x).$$
Then, since 
$\E_{P^x}\left[[N^h,N^h]_1\right]=\|h\|_{\GG(P^x)}^2$ is finite, we have
$$\E_{P^x}[N^h_1]\leq H(P^x|R^x) +\frac{1}{2}\|h\|_{\GG(P^x)}^2.$$
Repeating the 
same calculation with $-h$ and $\lambda h$ for
$\lambda>0$, 
 for all $h\in\HH(P^x)\cap\HH(R^x)$
\begin{equation}\label{G.E6}
\lambda\left|\E_{P^x}[N^h_1]\right|\leq H(P^x|R^x)+\frac{\lambda^2}{2}\|h\|_{\GG(P^x)}^2.
\end{equation}	
If $\|h\|_{\GG(P^x)}\neq 0$, we can take $\lambda = \sqrt{2H(P^x|R^x)}\|h\|_{\GG(P^x)}^{-1}$ and obtain
\begin{equation*}
\left|\E_{P^x}[N_1^h]\right|\leq\sqrt{2H(P^x|R^x)}
 \|h\|_{\GG(P^x)}.
\end{equation*}
Letting $\lambda \to \infty$ 
in \eqref{G.E6}, this inequality remains valid if 
$\|h\|_{\GG(P^x)}=  0$. So the linear form~$h\mapsto \E_{P^x}[N^h_1]$ is continuous on $\HH(P^x)\cap\HH(R^x)$. This set is dense in $\HH(P^x)$ since it contains the dense set of stair 
processes
\begin{equation*}
h :(t,\omega)\in[0,1]\times \Omega\mapsto \sum_{i=1}^k \sigma(X_t)h_i\ind_{]S_i,T_i]},
\end{equation*}
with $k\in\N$, $(h_i)_{1\leq i\leq k}\in\R^n$ and $S_i<T_i\leq S_{i+1}$ stopping times.
So,
$h\mapsto \E_{P^x}[N^h_1]$ 
extends linearly in a unique continuous way
to
$\HH(P^x)$. By Riesz representation theorem, there exists a process $\zeta^x\in\HH(P^x)$ dual to 
this linear form, i.e
\begin{equation*}
\E_{P^x}\left[\int_0^1{\langle 
\mathrm df, \mathrm d^{R^x}_mX_t\rangle}\right] = \E_{P^x}\left[\int_0^1{\langle 
\mathrm df, \mathrm d
A_t(\zeta^x_t\otimes \cdot)\rangle}\right].
\end{equation*}
In conclusion, under $P$, $X$ is a semi-martingale with quadratic variation $A$ and drift $B+\hat{B}$.

\end{proof}

We remark that using 
the classical Girsanov theory we could only 
have proved that,~$P^x$-almost surely,
$A_t(\zeta^x\otimes\cdot)<+\infty$.
From now on, $\zeta^x_t$ will be identified, as a vector field, with the drift $\hat{B}_t = A_t(\zeta^x\otimes\cdot)$. We show in Section \ref{sec:kin} that it is the Nelson forward stochastic velocity of $P^x$.

Léonard's approach to Girsanov theory also gives us an expression of the density of
 $P^x$ with respect to $R^x$ in terms of $\zeta^x$.
This will be essential
for the proof of
Theorem \ref{theo:NavierStokes}.

\begin{theo}[Density in terms of velocity]
\label{theo:GirsanovDensity}
With the notation of Theorem \ref{theo:Girsanov}, 
for $P_0$-almost every
$x \in M$, the density of $P^x$ is given by
\[
\frac{\mathrm dP^x}{\mathrm dR^x}= 
\ind_{\left\{\frac{\mathrm dP^x}{\mathrm dR^x}>0\right\}}\exp\left(\int_0^1\langle\zeta^x_t, \mathrm d^{P^x}_mX_t\rangle -\frac{1}{2}\int_0^1
\|\zeta^x_s\|^2\, \mathrm ds\right).
\]
\end{theo}

\begin{proof}[Sketch of the proof]
The proof is divided in three parts. Firstly, we prove a change of measure formula for stopped processes. This is the following well-known argument. 
We define the sequence 
$(\sigma_k)_{k \geq 1}$ of stopping times
 by
\begin{equation*}
\sigma_k =\inf\left\{t\in[0,1] : A_t(\zeta^x\otimes\zeta^x)\geq k\right\},
\end{equation*}
where, since we are thinking
on subsets of $[0,1]$, 
we use the convention that
the infimum of the empty set is $1$.
These stopping times localise the 
semi-martingale 
\begin{equation*}
N_t =\int_0^t\langle \zeta^x_s, \mathrm d^{R^x}_mX_s\rangle,\, 0\leq t\leq 1.
\end{equation*}
Let $R^{\sigma_k}$ denote
the law of $X_{\cdot \wedge \sigma_k}$ when
$X$ follows the law $R^x$
and let 
$\EE(N)_{\sigma_k}$
denote the stochastic exponential
of $N$ at the time $\sigma_k$.
Hence, the measure $Q_k = \EE(N)_{\sigma_k}R^{\sigma_k}$ is a probability measure satisfying the martingale problem $\MM\PP\big((B+\hat{B})_{\cdot\wedge\sigma_k}, A_{\cdot\wedge\sigma_k}\big)$.
As a second step, using the additional assumption 
that $P$ is equivalent to $R$, 
we prove the theorem. Here, the key argument is a uniqueness property satisfied by the reflected Brownian motion: $R^x$ is the unique measure in $\MM\PP(B,A)$ absolutely continuous with respect to $R^x$ starting from $R_0$. Property gives us the density of $P^{\sigma_k}$. The equivalence assumption is used to have $\sigma_k\to+\infty$ $R^x$-a.s and obtain the density of $P^x$.  We finish with a regularisation argument. The measure $P^x_n = (1-\frac{1}{n})P^x+\frac{1}{n}R^x$ is equivalent to $R^x$ and converge to $P^x$ in a sufficiently strong sense to obtain the result at the limit.
\end{proof}

\begin{remark}[Entropy and kinetic energy]
\label{rem:EntropyKineticEnergy}
The proof of Theorem 
\ref{theo:GirsanovDensity}
also
works for
$P$ and $R$ instead of $P^x$
and $R^x$.
As a consequence,
we would obtain that 
if  $H(P|R)<\infty$,
\begin{equation*}
H(P|R) = H(P_0|R_0) +\frac{1}{2}\E_P\left[\int_0^1
\|\zeta_t\|^2\, \mathrm dt\right],
\end{equation*}
where $\zeta$
can be seen as a forward
stochastic velocity
as in \eqref{eq:ForwardStocVelocity}
by using $P$ instead of $P^x$.
This formula for the entropy
 makes a parallel between the Brenier problem, 
as the minimisation of a 
\emph{classical} kinetic energy, and Brenier-Schrödinger problem, as the minimisation of a \emph{stochastic} kinetic energy in Nelson's sense. The advantage of an entropy formulation of the problem is the convex optimisation tools. 

\end{remark}

\section{Proof of the Navier-Stokes equations and the continuity equation}\label{sec:kin}
\setcounter{equation}0


This section will be devoted
to the proof of Theorem \ref{theo:NavierStokes}
and Theorem \ref{theo:ContinuityEquation}.
To prove
Theorem \ref{theo:NavierStokes}
we will first prove its 
`forward velocity' counterpart
in Corollary 
\ref{cor:ForwardNavierStokes}.
Following~\cite{ACLZ}, the idea
is to compare the density obtained
from Theorem \ref{theo:GirsanovDensity}
and the density
from the definition of a regular solution 
\eqref{eq:RegularSolutionForm}.
This is done in the following lemma.
We recall that~$\zeta^x$ is the one obtained
in Theorem \ref{theo:Girsanov}.
%

\begin{lemme}[Comparison of densities]
\label{lem:ZetaPsiRelation}
Suppose
that $P$ is a regular
solution of \eqref{BS}. 
Then, for $P_0$-almost every $x\in M$ and
for every $t\in [0,1]$,
\[
\langle \zeta^x_t,
 \mathrm d_m^{R^x}X_t
\rangle -
\frac{1}{2}\left\|\zeta_t^x
\right\|^2 \mathrm d t = \ind_\SS(t)\theta_t + 
p_t \mathrm dt + \mathrm d\psi^{x}_t(X_t),\, P^x\text{-almost surely}.
\]
\end{lemme}

\begin{proof}
The idea of the proof is to compare two expressions of the density of $P^x$ with respect to $R^x$, 
the first given by 
Theorem \ref{theo:GirsanovDensity} and the second by the definition of a regular solution. 
On the one hand,
from Theorem \ref{theo:GirsanovDensity}, 
the density is
\begin{equation*}
\frac{\mathrm dP^x}{\mathrm dR^x} = 
\exp\left(\int_{[0,1]}\langle\zeta^x_t,
\mathrm d_m^{P^x}X_t
\rangle -\frac{1}{2}\int_{[0,1]}\left\|\zeta^x_t\right\|^2\, \mathrm dt\right),\, P^x\text{-a.s.}
\end{equation*}
Then we restrict the density to 
$\FF_t = 
\sigma((X_u)_{u \in [0,t]})$ by using that
\begin{equation*}
\frac{\mathrm dP^x_{[0,t]}}{\mathrm dR^x_{[0,t]}} = \E_{R^x}\left[\left.
\frac{\mathrm dP^x}{\mathrm dR^x}\right|\FF_t\right].
\end{equation*}
For all 
$t \in [0,1]$
 we have 
\begin{equation*}
\frac{\mathrm dP^x_{[0,t]}}{\mathrm dR^x_{[0,t]}}
 = \exp\left(\int_0^t\langle\zeta^x_s,
  \mathrm d_m^{P^x}X_s\rangle -\frac{1}{2}\int_0^t\left\|\zeta^x_s\right\|^2\,\mathrm ds\right),\, P^x\text{-a.s.}
\end{equation*} 
On the other hand, from the definition of 
a regular solution, we know that $P$ has the form~\eqref{eq:RegularSolutionForm}. By disintegration, for $R_0$-almost every $x\in M$, 
\begin{equation*}
\frac{\mathrm dP^x}{dR^x} = 
\exp\left(\eta(x,X_1)+ \sum_{s\in\SS}
\theta_s(X_s)+ \int_\TT p_t(X_t)\, \mathrm dt\right),
\, R^x\text{-a.s.}
\end{equation*}
Then, conditioning with respect to $\FF_t$ and using the Markov property of $R^x$, 
\begin{align*}
\frac{\mathrm dP^x_{[0,t]}}{\mathrm dR^x_{[0,t]}} 
&= \E_{R^x}\left[\left.\exp\left(\eta(x,X_1)+ \sum_{s\in\SS}\theta_s(X_s)+ \int_\TT p_r(X_r)\, 
\mathrm dr\right)\right|\FF_t\right]\\
&= \exp\left(\sum_{s\leq t}\theta_s(X_s)+ \int_{\TT\cap[0,t]} p_r(X_r)\, \mathrm dr\right)\\
&\qquad\times\E_{R^x}\left[\left.\exp\left(\eta(x,X_1)+ \sum_{s>t}\theta_s(X_s)+ \int_{\TT\cap]t,1]} p_r(X_r)\, \mathrm dr\right)\right|\FF_t\right]\\
&=\exp\left(\sum_{s\in\SS, s\leq t}\theta_s(X_s)+ \int_{\TT\cap[0,t]} p_r(X_r)\, 
\mathrm dr + \psi^x_t(X_t)\right),\, R^x\text{-a.s.}\\
\end{align*}
We confront both expressions for
$\frac{\mathrm dP^x_{[0,t]}}{\mathrm dR^x_{[0,t]}} $
and conclude.

\end{proof}

\begin{theo}[Hamiltonian equation for the potential]
\label{theo:Hamilton}
Assume that $P$ is 
a forward regular solution
of \eqref{BS}.
Then, for $P_0$-almost 
every $x\in M$, the function $\psi^x$ defined by 
\eqref{eq:psi}
 is a classical solution of 
\begin{equation}\label{eq:HamiltonJacobiForward}
\left\{\begin{aligned}
&\partial_t\psi^x_t -\frac{1}{2}\Delta\psi^x_t +
\|\nabla\psi^x_t\|^2+\ind_\TT(t)p_t=0, && t \in [0,1) \setminus \SS,\\
&\psi^x_t-\psi^x_{t^-} = -\theta_t, && t\in \SS,\\
&\langle\nabla\psi^x_t(z),\nu_z\rangle = 0, && z\in\partial M,\\
&\psi^x_1=\eta(x,\cdot), && t=1,\\
\end{aligned}\right.
\end{equation}
which is a second-order Hamiltonian equation.

Moreover, 
for
$P_0$-almost every
$x \in M$,
\[
\zeta_t^x
= \nabla\psi^{x}_t(X_t),\, 
\mathrm{d t} \otimes P^x\text{-almost surely.}
\]
\end{theo}

\begin{proof}
According to Theorem \ref{theo:Girsanov}, we have 
$\mathrm d_m^{P^x} X_t = \mathrm d_m^{R^x}X_t - 
\zeta_t \mathrm dt$, $P^x$-almost surely. 
By Lemma \ref{lem:ZetaPsiRelation},
for $P_0$-almost every 
  $x\in M$ and for all $t \in [0,1]$,
\begin{equation*}
\mathrm d\psi^x_t(X_t) = 
\langle \zeta^x_t, \mathrm d_m^{P^x}X_t\rangle 
+\left(\frac{1}{2}
\|\zeta^x_t\|^2 -p_t(X_t)\right)
\mathrm dt - \ind_\SS(t)\theta_t(X_t),\, 
P^x\text{-almost surely.}
\end{equation*}
On the other hand, since
$\psi^x$ is regular enough, 
the semi-martingale $(\psi^x(X_t))_{t\in [0,1]}$ satisfies the Itô formula. For all $t \in [0,1]$, 
we have
\begin{align*}
\mathrm d\psi^x_t(X_t) = 
&[\psi^x_t-\psi^x_{t^-}](X_t) + \left\langle\nabla\psi^x_t(X_t), \mathrm d_m^{P^x}X_t\right\rangle +
\left\langle\nabla\psi^x_t(X_t), \zeta^x_t\right\rangle \mathrm dt\\  
&+\left\langle\nabla\psi^x_t(X_t),\nu_{X_t}\right\rangle \mathrm dL_t + 
\left(\frac{1}{2}\Delta +\partial_t\right)
\psi^x_t(X_t) \mathrm dt ,\, P^x\text{-a.s.}
\end{align*}
The Doob-Meyer decomposition of 
a semi-martingale allows the following
identifications using the  previous equations.
\begin{equation*}
\left\{\begin{aligned}
&\zeta^x_t = \nabla\psi^x_t (X_t),\, && 
\mathrm dt \otimes \mathrm dP^x(X)\text{-a.s.}\\
& - \ind_\SS(t)\theta_t(X_t) = [\psi^x_t-\psi^x_{t^-}](X_t),\, 
&& \mathrm dP^x(X)\text{-a.s.}\\
& \frac{1}{2}\|\zeta^x_t\|^2 -p_t(X_t) = \langle\nabla\psi^x_t(X_t), \zeta^x_t\rangle + \left(\frac{1}{2}\Delta +\partial_t\right)\psi^x_t(X_t),\, 
&& 
\mathrm dt \otimes \mathrm d P^x\text{-a.s.}\\
& \langle \nabla \psi^x_t(X_t),\nu_{X_t}\rangle\ind_{\partial M}(X_t) = 0,\, && 
\mathrm dL_t(X) \otimes
\mathrm dP^x(X)\text{-a.s.}\\
\end{aligned}\right.
\end{equation*}
We complete the proof
by using the covering property of $X$ under $P^x$ so that
\begin{equation*}
\left\{\begin{aligned}
&\zeta_t^x = \nabla\psi^x_t (X_t),\, && \mathrm dt 
\otimes \mathrm dP^x(X)\text{-a.s.},\\
& -\theta_t(z) = [\psi^x_t-\psi^x_{t^-}](z),\, && t\in\SS,  z\in M,\\
&\left(\frac{1}{2}\Delta +\partial_t\right)\psi^x_t(z) +\frac{1}{2}\|\nabla\psi^x_t(z)\|^2 +\ind_\TT(t)p_t(z)=0,\, && t\in[0,1)\setminus\SS, z\in M,\\
& \langle \nabla \psi^x(z),\nu_{z}\rangle = 0,\, &&  z\in\partial M.\\
\end{aligned}\right. 
\end{equation*}
\end{proof}

The function $\psi^x$ plays the role of a scalar potential of $\zeta^x$. 
The previous theorem tells us that,  in fact,  the randomness for $\zeta^x$ can be thought of
as coming  only
from the position $X_t$.
Recall that $\nabla$ denotes
the covariant
derivative and 
$\square$
denotes the 
de Rham-Hodge-Laplace operator
$-(\mathrm d\delta + \delta \mathrm d)$
by identifying
vector fields with
one-forms.  
Using the notation of
Section \ref{sec:ResultsKinematic},
define the random times
\[
\tau_t =\frac{1}{2}\sup
\left\{h\geq0 : 
X_{t+s}\in U(X_t)\mbox{ 
for every } s \in [0,h] \right\}
\]
which are strictly positive.

\begin{corol}[Forward stochastic velocity]
\label{cor:ForwardNavierStokes}

Suppose that $P$
is a forward regular solution
of \eqref{BS}.
 Then, 
for $P_0$-almost every
$x \in M$, 
there exists
a measurable function
\[\vf{x}_{\hspace{-1mm}}
\hspace{1mm}:[0,1]\times \Omega
\to TM\]
such that
$t \mapsto
\vf{x}_{\hspace{-1mm}t}
\hspace{-1mm}
(\omega)$ is 
right-continuous and has left
limits for every
$\omega \in \Omega$ and
such that
for every $t \in [0,1]$
we have
that, for
$P^x$-almost every $X$,
\begin{equation}\label{eq:ForwardStocVelocity}
\lim_{h\to 0^+}\frac{1}{h}\E_{P^x}
\left[\left. 
\overrightarrow{
X_{t}X_{t+h 
\wedge \tau_t}}
\right|X_{[0,t]}\right]
=\vf{x}_{\hspace{-1mm}t}
\hspace{-1mm}(X).
\end{equation}
Let $V_t^x = \nabla 
\psi_t^x$.
Then, for $P_0$-almost every
$x \in M$,
 \[
 \vf{x}_{\hspace{-1mm}t}
\hspace{1mm}=
V_t^x(X_t),
 \quad 
 P^x\text{-almost surely}.\]
Moreover,
for $P_0$ almost all $x\in M$,
the time-dependent
vector field
$V^x$ satisfies 
\[
\left\{\begin{aligned}
&\left(\partial_t  +\nabla_{V^x}\right) V^x =
-\frac{1}{2}\square V^x-\ind_\TT(t)\nabla p, && 
t \in [0,1)\setminus 
\SS, \\
&V^x_{t}- V^x_{t^-} = 
-\nabla \theta_t, && t\in S,\\
&\langle V^x(z),\nu_z\rangle = 0, && z\in\partial M,\\
&V^x_1=\nabla
\big(\eta(\cdot,x)\big), && 
t=1.\\
\end{aligned}\right.
\]

\end{corol}
 
\begin{proof}
Theorem \ref{theo:Girsanov}
tells us that,
$P^x$-almost surely,
$\mathrm dX_t = 
\mathrm d_m^{P^x}X_t + \zeta_t^x
\mathrm dt + \nu_{X_t}
\mathrm dL_t$. 
Moreover, we know by
Theorem \ref{theo:Hamilton} that we
may choose a version of $\zeta^x$ that
is càdlàg in $t$.
Applying Itô's formula to the 
logarithm 
$\log$, 
and taking the limit
by using the right-continuity 
of $\zeta^x$
we show that
the limit defining
$\vf{x}_{\hspace{-1mm}t}\hspace{1mm}$
exists and that
$\vf{x}\hspace{-1mm}(X)= \zeta^x= \nabla\psi^{x}_t(X_t)$. 
The rest is a consequence of Theorem 
\ref{theo:Hamilton} by taking the 
gradient of
\eqref{eq:HamiltonJacobiForward}
and using that
$(\mathrm d\delta + \delta \mathrm d)$
commutes with the exterior derivative $\mathrm d$.
\end{proof}

%

Now we are ready to give the proof of Theorem 
\ref{theo:NavierStokes}.

\begin{proof}[Proof of Theorem  \ref{theo:NavierStokes}]
There is a strong link between
$ \vf{x}_{\hspace{-1mm}t}$, 
the forward stochastic velocity obtained
in Corollary \ref{cor:ForwardNavierStokes}, and 
the backward stochastic velocity 
to be obtained here. This is
achieved through the time reversal transformation. 
Let 
$\mathrm{rev}:\Omega \to \Omega$
be the time reversal
transformation
defined by
$\mathrm{rev}(\omega)_t = 
\omega_{1-t}$. 
Let $P^*=\mathrm{rev}_*P$
be the pushforward of $P$
by the map $\mathrm{rev}$.
Then, 
the limit defining 
the forward stochastic velocity
is related to the sought limit for
the backward stochastic velocity by
\begin{equation}
\label{eq:ForwardVsBackward}
\vb{y}_{\hspace{-1mm}t}
\hspace{1mm} = -\vf{y}^{P^*}_{1-t}\circ\, 
\mathrm{rev},
\end{equation}
where $\vf{y}^{P^*}$
denotes the forward stochastic
velocity associated to $P^*$.
We just need to notice that $P^*$
satisfy the requirements
of Corollary \ref{cor:ForwardNavierStokes}.
Since the reference measure $R$ is reversible, we have 
$\mathrm{rev}_* R=R$ so that
\eqref{eq:RegularSolutionForm} becomes
\begin{equation*}
\mathrm d P^*(X) = 
\exp\left(\eta^*(X_0,X_1) + \sum_{s\in\SS*}\theta^*_s(X_s)+ \int_{\TT^*} p^*_t(X_t)\, 
\mathrm dt\right) \mathrm d R(X), 
\end{equation*}
where $\eta^*(x,y) = \eta(y,x)$, $\SS^*=\{1-s :s\in\SS\}$, $\theta^*_s = \theta_{1-s}$, 
$\TT^* = \{1-t : t\in\TT\}$ and~$p^*_t = p_{1-t}$. 
The function $\psi$ from \eqref{eq:psi} for $P^*$ equals the function
$\varphi$ from 
\eqref{eq:phi} for $P$ or, more precisely,
\[
\varphi^y_t(z) = \log\E_{R^y}\left[\left.\exp\left(\eta^*(y,X_1) + \sum_{s\in\SS^*, s>t}\theta^*_s(X_s) +\int_{\TT^*\cap]t,1]}p^*_r(X_r)\, 
\mathrm dr\right)\right| X_t=z\right].
\]
We may conclude
by Corollary \ref{cor:ForwardNavierStokes}
and \eqref{eq:ForwardVsBackward}.

\end{proof}

Now we give the proof of Theorem 
\ref{theo:ContinuityEquation}.
Notice that we will not use the regularity
of the solution.
In particular, there are no
$\eta$, $p$ and $\theta$ involved.

\begin{proof}[Proof of Theorem \ref{theo:ContinuityEquation}]

Let $f\in\CC^\infty(M)$, and $0\leq t\leq1$. On one hand, we have that
\begin{align*}
P_t(f)
&= \E_P[f(X_t)]\\
&= \int_M\E_{P^x}[f(X_t)]\, \mathrm d P_0(x) \\
&=\int_M \E_{P^x}\left[f(x) + \int_0^t\langle 
\mathrm df, \vf{x}_s\rangle\, \mathrm ds  + \frac{1}{2}\int_0^t\Delta f(X_s)\, \mathrm ds +\int_0^t\langle \mathrm df, \nu_{X_s}\rangle\, \mathrm dL_s\right]\, \mathrm d P_0(x) \\
&=P_0(f) + \int_0^t P_s\left(\langle 
\mathrm df, \stackrel{\hookrightarrow}{v}_s\rangle +\frac{1}{2}\Delta f\right)\, \mathrm ds 
+\E_P\left[\int_0^t\langle \mathrm df, \nu_{X_s}\rangle\,
\mathrm dL_s\right].
\end{align*}
On the other hand, for $P_1$-almost every $y\in M$, under the reversed law ${P^y}^*$, $X$ is a semi-martingale with drift~${\vf{y}^{P^*} 
\hspace{-3mm}\mathrm dt +\nu \mathrm dL^*}$ where the relation between~$\vb{y}$ and $\vf{y}^{P^*}
\hspace{-2mm}$ is given by~\eqref{eq:ForwardVsBackward} and $L_t(X) = L^*_{1-t}(X^*)$ for all $t \in [0,1]$, where
$X^* = \mathrm{rev}(X)$. We have 
\begin{align*}
&\E_{{P^y}^*}[f(X_{1-t})]\\
&\quad= \E_{{P^y}^*}\left[f(y) + \int_0^{1-t}\langle \mathrm df, \vf{y}_s^{P^*}
\hspace{-3mm}(X)\rangle\, \mathrm ds  + \frac{1}{2}\int_0^{1-t}\Delta f(X_s)\, \mathrm ds +\int_0^{1-t}\langle \mathrm df, \nu_{X_s}\rangle\, \mathrm dL^*_s(X)\right]\\
&\quad= \E_{{P^y}^*}\left[f(y) + \int_t^1\langle \mathrm df, \vf{y}_{1-s}^{P^*}\hspace{-1mm}(X)\rangle\, \mathrm ds  + \frac{1}{2}\int_t^1\Delta f(X^*_s)\, \mathrm ds +\int_t^1\langle \mathrm df, \nu_{X_s}\rangle\, \mathrm dL^*_{1-s}(X)\right]\\
&\quad= \E_{{P^y}^*}\left[f(y) - \int_t^1\langle \mathrm df, \vb{y}_s \hspace{-1mm}
(X^*)\rangle\, \mathrm ds  
+ \frac{1}{2}\int_t^1\Delta f(X^*_s)\, \mathrm ds +\int_t^1\langle \mathrm df, \nu_{X_s}\rangle\, \mathrm dL_{s}(X^*)\right].
\end{align*}
Hence, by disintegration along $P_1$, we have that
\begin{align*}
P_t(f)
&= \E_{P^*}[f(X_{1-t})]\\
&=\int_M \E_{{P^y}^*}[f(X_{1-t})]\, \mathrm d P_1(y) \\
&= P_1(f) + \int_t^1 P_s\left(\langle \mathrm df, -\stackrel{\hookleftarrow}{v}_s(X)\rangle +\frac{1}{2}\Delta f\right)\, \mathrm ds +
\E_P\left[\int_t^1\langle \mathrm df, \nu_{X_s}\rangle\, \mathrm dL_s\right].
\end{align*}
Before differentiating, we need to show that the terms with local time are regular enough. For $\varepsilon>0$, we denote $\partial^\varepsilon M$ the $\varepsilon$-tubular neighbourhood of $\partial M$. We have 
\begin{equation*}
\int_0^t\langle \mathrm df, \nu_{X_s}\rangle\, 
\mathrm dL_s = \lim_{\varepsilon\to 0}\frac{1}{2\varepsilon}\int_0^t\langle \mathrm df, 
\nu_{X_s}\rangle\ind_{X_s\in\partial^\varepsilon M}\, \mathrm ds.
\end{equation*}
Then, we obtain
\begin{align*}
\E_P\left[\int_0^t\langle \mathrm df, 
\nu_{X_s}\rangle\, \mathrm dL_s\right]
&=\frac{1}{2}\int_0^t \lim_{\varepsilon\to 0}\frac{1}{\varepsilon}\E_P\left[\langle \mathrm df,
 \nu_{X_s}\rangle\ind_{X_s\in\partial^\varepsilon M}\right]\, \mathrm ds\\
&=\frac{1}{2} \int_0^t\lim_{\varepsilon\to 0}\frac{1}{\varepsilon} P_s(\langle \mathrm df, \nu\rangle\ind_{\partial^\varepsilon M})\, \mathrm ds\\
&=\frac{1}{2} \int_0^t \text{\underbar{$P$}}_s(\langle \mathrm df, \nu\rangle),
\end{align*}
where $\text{\underbar{$P$}}_s$ denotes 
the normalised surface measure associated to $P_s$. 
It follows that, for all~${t\in[0,1]}$,
\begin{equation}
\partial_t P_t(f) = P_t(\langle\stackrel{\hookrightarrow}{v}_t, \mathrm df\rangle + \frac{1}{2}\Delta f) + \text{\underbar{$P$}}_t(\langle \nu, \mathrm df\rangle) = P_t(\langle\stackrel{\hookleftarrow}{v}_t, \mathrm df\rangle - \frac{1}{2}\Delta f) - \text{\underbar{$P$}}_t(\langle \nu, \mathrm df\rangle).
\end{equation}
Since $P$ is a 
satisfies 
$P_t = \mu_t$ for every $t\in\TT$,
the proof is complete.
\end{proof}

\section{Proof of the existence on homogeneous spaces}\label{ex:sec.2}
\setcounter{equation}0
In this section we prove
 Theorem \ref{theo:Homogeneous}.
The proof
 is inspired by 
\cite{ACLZ} which is, in turn, inspired
by~\cite{Bre1}.
The idea is to find a path measure $Q$ of finite relative entropy and satisfying the marginal conditions. The candidate for such a measure is 
\begin{equation}\label{ex.E1}
Q=\int_{M^3}{R(\cdot |X_0=x, X_{1/2}=z, X_1=y)\, \sigma(\mathrm dx \mathrm dz \mathrm dy)},
\end{equation}
with $\sigma(\mathrm dx \mathrm dz \mathrm dy)= \pi(\mathrm dx \mathrm dy) vol(\mathrm dz)$ 
belongs to $\mathcal P(M^3)$. It extends the result in \cite{ACLZ} of existence on the torus, using the same property of invariance of the Brownian motion and the Riemannian volume, under isometries.

\begin{prop}[Constraints and entropy]
\label{prop:Q}
The path measure $Q$ satisfies the marginal and endpoint constraints
\[P_t=vol, \forall 
t \in [0,1]\quad
\mbox{ and } \quad P_{01}=\pi.\] In addition, 
if $H(\pi|vol \otimes vol)<\infty$ then $H(Q|R)<\infty$.
\end{prop}
	
\begin{proof}
First, remark that, since $R$ is a Markov measure, we have
\begin{align*}
R(\cdot|X_0=x, X_{1/2}=z, X_1=y)=R(X_{[0,1/2]}&
\cdot |X_0=x, X_{1/2}=z)\\
&\times R(X_{[1/2,1]}\in \cdot| X_{1/2}=z, X_1=y).
\end{align*}
Now, let us check the endpoint constraints. 
For measurable
subsets $A$ and $B$ of $M$,
\begin{align*}
Q_{01}(A\times B)
&=Q(X_0\in A, X_1\in B)\\
&=\int_{M^3}R(X_0\in A|X_0=x, X_{1/2}=z)R(X_1\in B| X_{1/2}=z, X_1=y)\, 
\sigma(\mathrm dx \mathrm dz \mathrm dy)\\
&=\int_{M^3}\ind_A(x)\ind_B(y)\, 
\sigma(\mathrm dx \mathrm dz \mathrm dy)\\
&=\sigma(A\times M\times B)\\
&=\pi(A\times B).
\end{align*}
So $Q_{01}=\pi$.
Then, we prove that $Q_t$ is 
invariant 
under isometries
for all $t$. Let $t
\in [0,1/2]$ and
let $f$ be a bounded measurable function on $M$. We have 
\begin{align*}
\int_M f\, \mathrm dQ_t 
&=\int_{M^3}\E_R\left[
f(X_t)\,
\left|\,X_0=x, X_{1/2}=z \right.\right]\, \sigma(\mathrm dx \mathrm dz \mathrm dy)\\
&=\int_{M^2}\E_R\left[f(X_t)
\, \left|\,X_0=x, X_{1/2}=z
\right.\right] vol(\mathrm dx) vol(\mathrm dz).
\end{align*}
For every isometry $g$
of $M$, using the invariance in law of the Brownian motion and the invariance of the Riemannian volume
measure under isometry, we have 
that
\begin{align*}
\int_M f\circ g\, \mathrm dQ_t
&=\int_{M^2}\E_R\left[f\circ g(X_t)
\,\left|\,X_0=x, X_{1/2}=z
\right.\right]\, vol(\mathrm dx) vol(\mathrm dz)\\
&=\int_{M^2}\E_R\left[f(X_t)
\,\left|\,X_0=g(x), X_{1/2}=g(z)
\right.\right]\, vol(\mathrm dx) vol(\mathrm dz)\\
&=\int_{M^2}\E_R\left[f(X_t)
\,\left|\,X_0=x, X_{1/2}=z
\right.\right] vol(\mathrm dx) vol(\mathrm dz)\\
&=\int_M f\, \mathrm dQ_t. 
\end{align*} 
Since $vol$ is the unique 
probability
measure on $M$ that
is invariant
under isometries
(see, for instance, 
\cite[Proposition 476C]{FREM}),
we obtain $Q_t=vol$. 
The result is obtained, mutatis mutandis, for 
$t \in [1/2,1]$.
Then, $Q$ satisfies the marginal constraint.
		
We have now to prove
that $H(Q|R)<\infty$. Denote 
\begin{equation*}
Q_{0,1/2,1}=Q(X_0\in
\cdot, X_{1/2}\in \cdot, X_1\in \cdot)
\end{equation*} 
and $Q^{xzy}=Q(\cdot|X_0=x,X_{1/2}=z,X_1=y)$,
and similarly for $R$. We have,
by using 
the chain rule for the entropy
\cite[Theorem C.3.1]{Dupuis},
\begin{align*}
H(Q|R)
&= H(Q_{0,1/2,1}|R_{0,1/2,1}) + \int_{M^3}H(Q^{xzy}|R^{xzy})\, Q_{0,1/2,1}(\mathrm dx 
\mathrm dz \mathrm dy)\\
&= H(\sigma |R_{0,1/2,1})\\
&= H(\sigma_{01} |R_{01}) + \int_{M^2}H(\sigma^{xy}|R^{xy}_{1/2})\, \sigma_{01}(\mathrm dx \mathrm dy)\\
&= H(\pi |R_{01}) + \int_{M^2}H(vol|R^{xy}_{1/2})\, \pi(\mathrm dx \mathrm dy),
\end{align*}
where $R_{1/2}^{xy}$
is
the law at time $t=1/2$ 
of the Brownian bridge
between $x$ at time $t=0$
and~$y$ at time $t=1$.
By definition of the relative entropy, we have 
\begin{equation*}
H(vol|R^{xy}_{1/2})=\int_M \log
\left(\frac{\mathrm d vol}{\mathrm d 
R^{xy}_{1/2}}\right)\, vol(\mathrm dz) .
\end{equation*}
We denote by $p$ the heat kernel on $M$. We have 
\begin{equation*}
\frac{\mathrm d R^{xy}_{1/2}}{\mathrm d vol}(z)=\frac{p_{1/2}(x,z)p_{1/2}(z,y)}{p_1(x,y)}.
\end{equation*}
This quantity is continuous in $x$, $y$ and $z$. As $M$ is compact, the density can be bounded uniformly in the three variables. So the relative entropy~$H(Q|R)$ is finite if and only if~$H(\pi |R_{01})$ is finite
which, since the density
of $R_{01}$ with respect to $vol \otimes vol$
is continuous and strictly positive,
is equivalent to 
$H(\pi |vol \otimes vol)<\infty$.
\end{proof}

The homogeneity of $M$ 
seems to be important
to show that, at a fixed time, 
the law of the Brownian bridge between two independent uniformly distributed random variables is the uniform measure $vol$. It 
is not clear and it
would be interesting to understand if this holds or not on a non-homogeneous space. 

\begin{proof}[Proof of Theorem \ref{theo:Homogeneous}]
By Proposition \ref{prop:Q},
if the entropy of $\pi$ is finite,
there exists
a measure $Q$
that satisfies the constraints
of the problem \eqref{BS}
and has finite entropy with respect to $H$.
We conclude by the strict convexity
of the entropy and the convex constraints.

On the other hand, if $Q$ is the unique solution
then, in particular,
$Q_{01}$ has finite entropy with respect to
$R_{01}$. Since $R_{01}$
has a continuous and strictly positive density
with respect to $vol \otimes vol$, we also have
that 
$H(Q_{01}|vol \otimes vol)< \infty$.
\end{proof}

\section{Proof of the existence for quotient spaces}
\label{ex:sec.3}
\setcounter{equation}0
The goal of this section is to
prove
Theorem  \ref{th:ToQuotient}
and to give some examples of 
the existence of solutions to the Brenier-Schrödinger problem.
Theorem  \ref{th:ToQuotient}
describes
a relation between
the Brenier-Schrödinger problem
on compact Riemannian manifolds
and on some quotients of these.
For instance, 
we want to see the 
\textit{$n$-hypercube}
as a quotient of a
\textit{flat $n$-dimensional torus} (see Figure
\ref{fig:PathTorus} and~\ref{fig:PathRectangle}
where the reflections are along
the dotted lines) or
a positively curved \textit{$n$-ball}
as a quotient of the 
\textit{$n$-sphere}.

\begin{figure}[h]
\centering
\begin{minipage}{.5\textwidth}
  \centering
  \includegraphics[width=.9\linewidth]{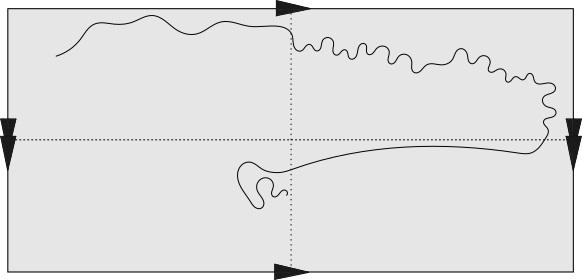}
  \captionof{figure}{A path in the torus.}
  \label{fig:PathTorus}
\end{minipage}%
\begin{minipage}{.5\textwidth}
  \centering
  \includegraphics[width=.9\linewidth]{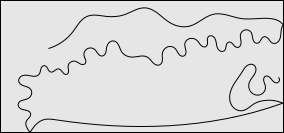}
  \captionof{figure}{The projected path in the rectangle.}
  \label{fig:PathRectangle}
\end{minipage}
\end{figure}

%

We begin by giving a proof
of the existence
of the Riemannian structure
on a quotient by a
reflection group. This
is similar to what happens
on $\mathbb R^n$
where the
theory of reflection
groups is well-known
(see, for instance, \cite{Hum}).

\begin{proof}[Proof of Lemma
\ref{lem:RiemannianMwithCorners}]

The topological structure of
$N =M/G$
is
induced by the quotient map
$q:M \to N$.
Given a reflection group $G$, we shall make of $N$ a manifold with
corners in the following way. 
Let $y \in N$ and take any
$x \in M$ with $q(x) = y$.

If $G_x = \{e\}$ then there exists
an open neighborhood
$V$ of $x$ such that 
${g  V \cap h V = \emptyset}$
for every $g \neq h$ in $G$. 
Since $q$ is open and $q|_V$ is injective and continuous we have that~${q|_V: V \to q(V)}$
is an homeomorphism and we 
can assume that (by taking
a smaller~$V$ if necessary) 
$V$ is diffeomorphic to an
open subset of $(0,\infty)^n$.
This gives an atlas to the open
set of points that can be written
as $q(x)$ with
$G_x = \{e\}$. We
can even define a metric
on this open set with the help
of these
$q|_V$.

If $G_x \neq \{e\}$ we consider the
exponential map
\[\exp_x:W \subset T_x M \to 
V \subset M\]
on an open neighborhood $W$
of $0 \in T_xM$ 
invariant under $\mathbb G_x$
such that $\exp_x|_W$
is a diffeomorphism onto its image
$V$.
Moreover, by choosing 
$V$ small enough we assume that~$g V \cap V = \emptyset$
for every $g \notin G_x$. 
Since 
\begin{equation}
\label{eq:isom}
g\exp_x(w) = 
\exp_x(\mathrm d g_x w)
\end{equation}
for $g \in G_x$ and $w \in T_xM$,
the open set $V$ is invariant under
$G_x$. Equation \eqref{eq:isom} tells
us that the action of
$G_x$ on $W$ (as $\mathbb G_x$)
is isomorphic to the action of 
$G_x$ on $V$. Then, we only need
to understand
\[W/\mathbb G_x.\]
But, since $\mathbb G_x$
is a reflection group, we know that
$T_x M /\mathbb G_x$ can be
identified
with a particular 
fundamental
domain of the action
of $\mathbb G_x$ on $T_x M$,
called closed chamber 
(see~\cite[Section~1.12]{Hum}),
and, in particular,
it has a structure
of a manifold with corners
so that~$W/\mathbb G_x$ inherits this
structure. Using $\exp_x$
we have given to the open set
$V/G_x \simeq 
q(V)$ the structure of a manifold
with corners. In fact, if
$C \subset T_xM$ is a closed chamber,
we have identified~${\exp_x(C \cap W)}$ with $q(V)$.
The latter identification
gives a Riemannian metric
to~$q(V)$
which is completely
characterised
by the isometric properties
required for
$q$.
 
\end{proof}

We will also need the following standard lemma 
whose proof we recall.

\begin{lemme}[Fundamental domain]
\label{lem:Distinguished}
Let $G$ be a finite group of
isometries of $M$.
Then, there exists an open subset $V$ of $M$
such that 
\begin{itemize}
\item $g V \cap h V = \emptyset$
for every $g \neq h$ in $G$ and
\item $vol\left(M\, 
\displaystyle \Bigg\backslash 
\bigcup_{g \in G} g V
\right) = 0 $ .
\end{itemize}
\begin{proof}
Let $x \in M$ such $G_x = \{e\}$
and define the set
\[V = \left\{y \in M: \,
\forall g \in G \setminus \{e\}, 
\, d(x,y)<
d(g x,y) \right\},\]
where $d$ is
the distance function on
the Riemannian manifold $M$.
Since $G$ is a group of isometries
we have that
\[h V = \left\{y \in M: \,
\forall g \in G\setminus \{h\}, \, d(h x,y)<
d(g x,y) \right\}.\]
We only need to see that,
for $\xi \neq \zeta$ in $M$, 
\[vol\left\{y \in M:\,
d(\xi,y) = d(\zeta,y)
\right\}=0.\]
This is
true since the map
$y \mapsto d(\xi,y) - d(\zeta,y)$
is smooth and regular
outside the cutlocus of $\xi$ 
and $\zeta$ and since every cutlocus
has $vol$-measure zero.

\end{proof}
\end{lemme}
Notice that, in particular,
for every $g \in G$ and 
$x \in gV$ the group
$G_x$ contains only the identity so that
$q|_{gV}$ is an isometry
onto its image.
There is an intuitive relation
between a Brownian motion
on $M$ and on its quotient by
a reflection group.

\begin{lemme}[Brownian motion under quotients]
\label{lem:QuotientBrownian}
Suppose that $G$ is a reflection
group of isometries of $M$.
Let $\{B^{x}_t\}_{t \geq 0}$
be a Brownian motion on $M$
starting at $x \in M$. Then
$\{q(B^{x}_t)\}_{t \geq 0}$
does not touch
the corner points almost surely and
\[\{q(B^{x}_t)\}_{t \geq 0}
\mbox{ is a reflected
Brownian motion on } 
 M/G.\]
 \begin{proof}
 The fact that
 $q(B_t^x)$ does not touch the corner points
 is a result of the following facts.
 The set $q^{-1}(\mathcal CN)$ is
 a finite union of submanifolds of
 dimension less or equal than~$n-2$ and
 the Brownian motion $B_t^x$ 
 almost surely
 does not touch submanifolds
 of dimension less or equal than $n-2$.

Now, for 
every $\varepsilon>0$, we consider
the $\varepsilon$-neighborhood
of the corner points,
 \[\mathcal N_\varepsilon =\left\{ x \in N : \, 
d(x,y)< \varepsilon \mbox{ for some } 
y \in \mathcal CN \right\}.\] 
Let $f:N \to \mathbb R$ be 
a smooth map such
that $\mathrm d f_x \nu_x = 0$
at every regular boundary
point $x$ and consider 
\[F = f \circ q\]
which can be seen
to be $C^2$ on $M\setminus q^{-1}(\mathcal CN)$. 
Let $\varepsilon>0$ and let
$F^\varepsilon:M \to \mathbb R$
be a $C^2$ function on $M$ that coincide
with $F$ outside of $\mathcal N_\varepsilon$.
 Then,
\[F^{\varepsilon}(B^x_t) - \int_0^t 
\Delta F^{\varepsilon}(B^x_s) \mathrm d s
\, \mbox{ is a martingale}\]
with
respect to the filtration
 $\left(B_s^x\right)_{s \in[0,t]}$ so that, if 
\[T_\varepsilon
= \inf \left\{t \geq 0:\, B^x_t \in 
q^{-1}\left(\mathcal N_\varepsilon\right)
 \right\},
\]
we have that
\[
F(B^x_{t \wedge T_\varepsilon} ) - 
\int_0^{t \wedge T_\varepsilon} 
\Delta F(B^x_s) \mathrm d s
\, \mbox{ is also a martingale}.
\]
By using that
\[\Delta F = 
\left(\Delta f \right) \circ q\]
we have proved that
\[f(q(B^x_{t\wedge T_\varepsilon})) - 
\int_0^{t\wedge T_{\varepsilon}} 
\Delta f(q(B^x_{s})) 
\mathrm d s 
\, \mbox{ is a martingale}\]
with
respect to the filtration
given by $\mathcal G_t=
\sigma(\left(B^x_s\right)_{s\in[0,t]})$. In particular,
since it is adapted
to the filtration given by
$\mathcal F_t =
\sigma(\left(q(B^x_s)\right)_{s\in[0,t]})$ and since
$\mathcal F_t \subset \mathcal G_t$,
it is also a martingale with
respect to this filtration.
Finally, since $q(B^x_t)$ does
not touch $\mathcal CN$, we can see 
that~${T_\varepsilon \uparrow \infty}$
as $\varepsilon \downarrow 0$
which completes the proof.

 \end{proof}

\end{lemme}

In the rest of this section and
for notational simplicity we
denote by $\sigma$, instead of~$vol$, the normalised volume measure
on $M$, and by $\tilde \sigma$,
the normalised volume measure
on~$N = M/G$.
Let $R$ be the law
of the Brownian motion on $M$
whose initial position
has law~$\sigma$
and let~$\tilde R$ be the law
of the reflected 
Brownian motion on $N$
whose initial position
has law $\tilde \sigma$.
We have the following
result.

\begin{lemme}[Image of the reversible
Wiener measure]
Denote by $q(R)$ the image
measure of $R$ by the map  induced
by $q$ from $C([0,1],M)$ to $C([0,1],N)$.
Then,
\[q(R) = \tilde R.\] 

\begin{proof}
By Lemma \ref{lem:QuotientBrownian},
$q(R)$ is the law of the Brownian
motion on $N$ whose
initial position
is distributed according to
$q_*\sigma$,
the image measure of $\sigma$
by $q$. It is enough,
then, to notice that
 $q_* \sigma = \tilde \sigma$.
 By Lemma \ref{lem:Distinguished}, 
 $\sigma
= 
\sum_{g\in G} \sigma|_{g V}$, so that
\[ q_* \sigma
= q_* \left( 
\sum_{g\in G} \sigma|_{g V}\right) 
=
\sum_{g\in G} q_*  
\left(\sigma|_{g V} \right).\]
We have that
\begin{equation}
\label{eq:MeasureComplement}
\tilde \sigma
\left(q(M \setminus \cup_{g \in G} 
\, \,g V)
\right)
=0.
\end{equation}
since the measure of $\partial N$ 
is zero and, on the
complement of
$q^{-1}(\partial N)$, 
the map $q$ is smooth
so that the image
of a set of measure zero
has also measure zero.
Since $q|_{g V}$ is an isometry
onto its image we have that
\[q_*  
\left(\sigma|_{g V} \right)
=
\mbox{ volume measure on }
N,\]
where we have used 
\eqref{eq:MeasureComplement}
which says that
$\tilde \sigma \left(
N \setminus q(gV) \right)
=0$. We obtain
\[q_* \sigma 
= \mbox{card}(G)
\left(\mbox{volume measure on }
N\right)\]
which, after normalising, concludes
the proof.
\end{proof}
\end{lemme}

We are ready to give the proof of
Theorem \ref{th:ToQuotient}.

\begin{proof}[Proof of Theorem
\ref{th:ToQuotient}]
As in the previous lemma,
we use the notation $\sigma$, instead of
$vol$, for the normalised volume measure
on $M$, and denote by $\tilde \sigma$,
the normalised volume measure
on $N = M/G$.

Let us prove
the first assertion
of Theorem \ref{th:ToQuotient}. 
Let $Q$ be a probability
measure on~$C([0,1],M)$
such that 
$Q_{01}=\pi$, $Q_t = \sigma$ for
every~${t \in [0,1]}$
and $H(Q|R) < \infty$.
We need to find 
a probability measure
$\tilde Q$ on $C([0,1],N)$
such that
$\tilde Q_{01}=(q\times q)_*\pi$, $\tilde Q_t 
= \tilde \sigma$ for
every~$t \in [0,1]$
and $H(\tilde Q|\tilde R) < \infty$.
Notice that
\[H(q(Q)| q(R))
\leq H(Q|R) < \infty.\]
Since $q(Q)$ satisfies
the marginal assumptions and since
$q(R) = \tilde R$, the proof
is completed by taking
$\tilde Q = q(Q)$.

Now, to prove
the second assertion we need to
write
$\tilde \pi$ as $q_* \pi$
for some nice $\pi$.
For this, we shall use Lemma 
\ref{lem:Distinguished}.
Since $H(\tilde \pi|\tilde R_{01}) 
< \infty$
we have that 
$H(\tilde \pi|\tilde \sigma \otimes 
\tilde \sigma) 
< \infty$. In particular,~$\tilde \pi$ gives
measure zero to
$N \times N
\setminus q(U) \times q(U)$.
For every $(g,h) \in G\times G$,
consider the map
\[
\left(q|_{gU} \times q|_{hU} 
\right)^{-1}:q(U) \times q(U)
\to gU \times hU\]
and consider the measure
\[\pi_{g,h} = 
\left(q|_{gU} \times q|_{hU} 
\right)^{-1}_* \tilde \pi
\]
which satisfies
\[(q \times q)_* \pi_{g,h}=  
\tilde \pi,\]
Nevertheless,
it does not satisfy
the marginal conditions.
Notice that, if
\[\alpha_{g,h} = 
\left(q|_{gU} \times q|_{hU} 
\right)^{-1}_* 
\left(\tilde \sigma \times
 \tilde \sigma \right)
 \quad
 \mbox{ and }
 \quad
 \sigma_g = \left(q|_{gU}\right)^{-1}_*
 \tilde \sigma
\]
then
\[\alpha_{g,h} = |G|^2
\left(\sigma \times \sigma \right)
|_{gU \times hU} \quad
\mbox{ and }
\quad
\sigma_g = |G|\, \sigma|_{gU}.
\]
Moreover, the first marginal
of $ \pi_{g,h}$
is $\sigma_g$ and its
second marginal is 
$\sigma_h$. Then, if
we define
\[\pi =\frac{1}{|G|^2}
\sum_{(g,h) \in G \times G}
\pi_{g,h}
\]
we may notice that
the first and second
marginals of $\pi$
are $\sigma$ and that
\[(q \times q)_* \pi
= \tilde \pi.
\]
We can also find its entropy
by integrating and obtain
that
\[H(\pi|\sigma \otimes \sigma)=
H(\tilde \pi|\tilde \sigma
\otimes \tilde \sigma).\]
Since $H( \pi|\sigma
\otimes \sigma)< \infty$
if and only if
$H( \pi| R_{01})< \infty$
and
$H(\tilde \pi|\tilde \sigma
\otimes \tilde \sigma)< \infty$
if and only if~$H(\tilde \pi| \tilde R_{01})< \infty$
we may conclude.

\end{proof}

We consider now some
simple examples of
quotient spaces where
Theorem \ref{th:ToQuotient} holds.
Almost all of these
will be quotients of
the flat two-dimensional torus
which we define now.
Let
$u$ and $v$ be two
linearly independent vectors of $\mathbb R^2$.
We will denote
by $\mathbb T_{u,v}$
the manifold
\[\mathbb T_{u,v} = 
\mathbb R^2 / \{au+bv:
a,b \in \mathbb Z\}\]
endowed
with the Riemannian metric
induced by $\mathbb R^2$.
We begin by describing
two examples that are actual
two-dimensional 
manifolds with boundary
(without corners).

\begin{example}[Cylinder]
Suppose that $u$ and $v$ are 
orthogonal. The map
\begin{align*}
\{\alpha u + \beta v:
(\alpha,\beta) \in [0,1]^2\}
&\to \{\alpha u + \beta v:
(\alpha,\beta) \in [0,1]^2\}\\
xu + yv &\mapsto xu + (1-y)v
\end{align*}
induces an isometry
of $\mathbb T_{u,v}$ and the quotient
space is isometric to the cylinder
\[\left\{z \in \mathbb C:
\, 2\pi |z| =|u|\right\} \times 
[0,|v|/2].\]
\end{example}

\begin{example}[Flat Möbius strip]

Suppose
that $|u| = |v|$. The map
\begin{align*}
\{\alpha u + \beta v:
(\alpha,\beta) \in [0,1]^2\}
&\to \{\alpha u + \beta v:
(\alpha,\beta) \in [0,1]^2\}\\
xu + yv &\mapsto yu + xv
\end{align*}
induces an isometry
of $\mathbb T_{u,v}$ and
the quotient space is isometric
to the flat Möbius strip
\[[0,\|u+v\|/2] \times [0,\|u-v\|/2]
/\sim\]
where $\sim$ is the identification
of the vertical sides in opposite
directions.
Figure \ref{fig:TorusAndMobius}
shows a representation
of the torus and the 
considered isometry
is the reflection along
the dotted diagonal.
Figure \ref{fig:Mobius}
shows the canonical
representation of the flat Möbius
strip as part
of (four times) the 
representation of the torus.

\end{example}

\begin{figure}[h]
\centering
\begin{minipage}{.5\textwidth}
  \centering
  \includegraphics[width=1\linewidth]{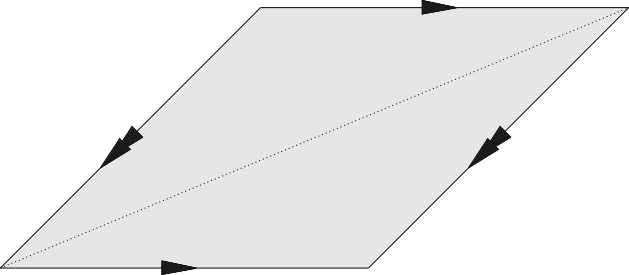}
  \captionof{figure}{The Möbius strip as a quotient.}
  \label{fig:TorusAndMobius}
\end{minipage}%
\begin{minipage}{.5\textwidth}
  \centering
  \includegraphics[width=1\linewidth]{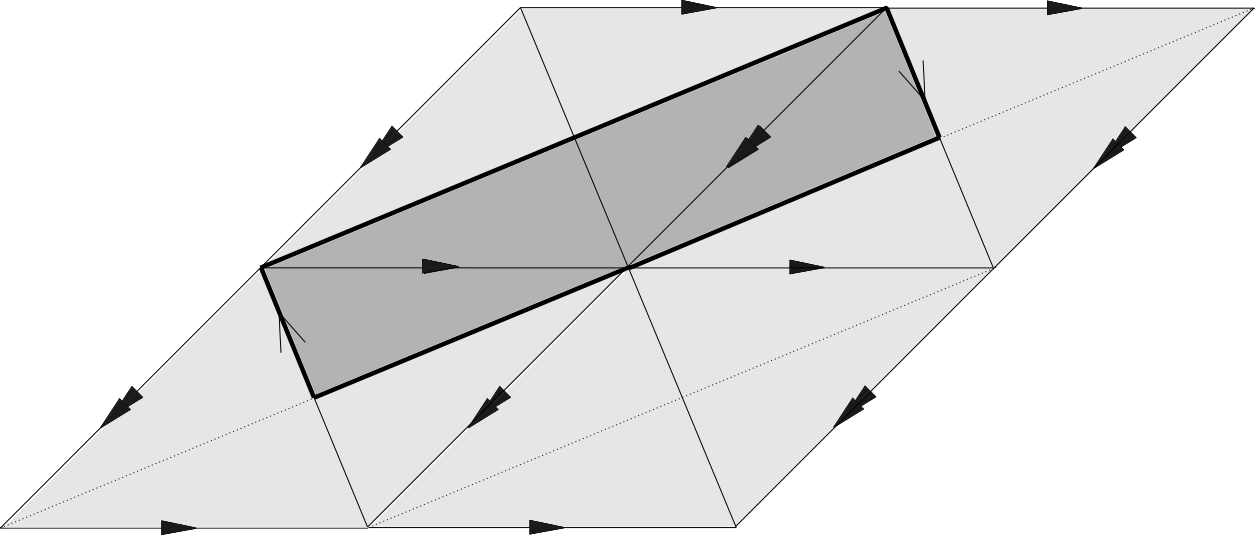}
  \captionof{figure}{The Möbius strip.}
  \label{fig:Mobius}
\end{minipage}
\end{figure}

The next four examples are 
two-dimensional manifolds 
with corners.

\begin{example}[Rectangle]
Suppose that $u$ and $v$ are 
orthogonal. The maps
\begin{align*}
\{\alpha u + \beta v:
(\alpha,\beta) \in [0,1]^2\}
&\to \{\alpha u + \beta v:
(\alpha,\beta) \in [0,1]^2\}\\
xu + yv &\mapsto xu + (1-y)v
\end{align*}
and
\begin{align*}
\{\alpha u + \beta v:
(\alpha,\beta) \in [0,1]^2\}
&\to \{\alpha u + \beta v:
(\alpha,\beta) \in [0,1]^2\}\\
xu + yv &\mapsto (1-x)u + yv
\end{align*}
generate a reflection group 
of isometries of 
$\mathbb T_{u,v}$ and the quotient
space is isometric to
\[[0,|u|/2]\times 
[0,|v|/2].\]
\end{example}

\begin{figure}[h]
\centering
  \includegraphics[width=.6\linewidth]{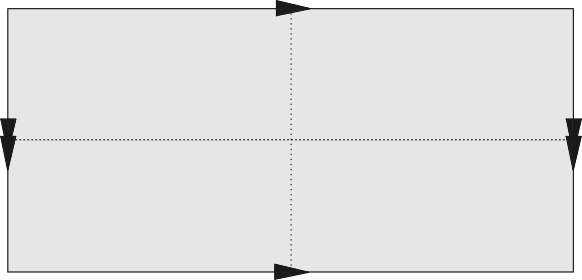}
  \captionof{figure}{A
  rectangle
	 as a quotient
  of the torus.}
  \label{fig:TorusAndRectangle}
\end{figure}

\begin{example}[Isosceles right triangle]

A $45^\circ$ right triangle can be seen
as a quotient of a square
by a reflection along
its diagonal. Using the previous
example, we can also 
see it as a quotient of a torus
(see Figure \ref{fig:TorusAndTriangleRectangle}).

\end{example}

\begin{figure}[h]
\centering
  \includegraphics[width=.4\linewidth]{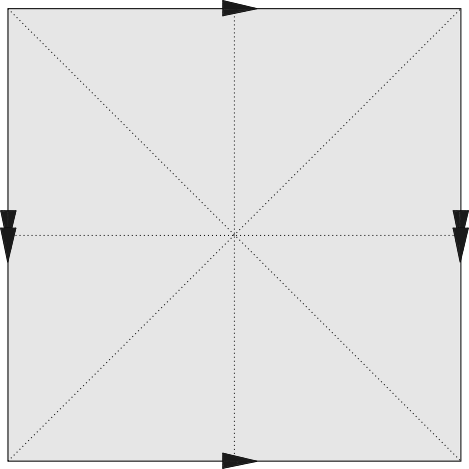}
  \captionof{figure}{A $45^\circ$
  triangle rectangle as a quotient.}
  \label{fig:TorusAndTriangleRectangle}
\end{figure}

\newpage

\begin{example}[Equilateral
triangle]

If $2 u\cdot v = \|u\| \|v\|$, the
torus $\mathbb T_{u,v}$
can be seen as a quotient of
an hexagon identifying
opposite sides as in Figure
\ref{fig:TorusAndTriangle}. 
Then, if we consider
the group generated by the
reflections along
the dotted lines in Figure
\ref{fig:TorusAndTriangle}
we can obtain
an equilateral triangle
as a quotient space.

\begin{figure}[h]
\centering
  \includegraphics[width=.4\linewidth]{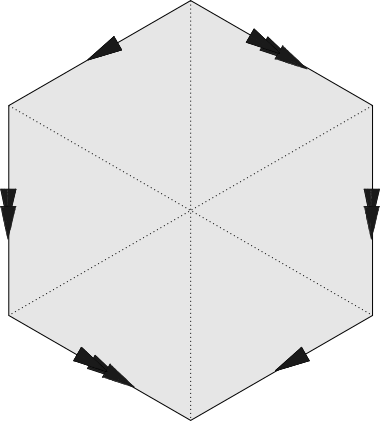}
  \captionof{figure}{An equilateral
  triangle as a quotient
  of the torus.}
  \label{fig:TorusAndTriangle}
\end{figure}

\end{example}

\begin{example}[$60^\circ$ right
triangle]
A $60^\circ$ right triangle
can be seen as a quotient
of the equilateral triangle
by a reflection. Using the previous
example we can see it also
as a quotient
of a torus.

\end{example}

Finally, as $n$-dimensional cases
we consider the following examples.

\begin{example}[$n$-hyperrectangle]
Let $a_1,\dots,a_n > 0$
and let 
$u_1,\dots,u_n$ be orthogonal
vectors in $\mathbb R^n$ such that
$\|u_i\| = a_i$ for any
$i \in \{1,\dots,n\}$.
 We may
consider the flat $n$-dimensional
torus
\[\mathbb T^n = 
\mathbb R^n / \{m_1 u_1 +
\dots + m_n u_n:
m_1,\dots,m_n \in \mathbb Z\}\]
and the group generated by
the reflections
induced by the family 
(indexed by $i \in \{1,\dots,n\}$)
of maps
\begin{align*}
\{\alpha_1 u_1 + \dots
+ \alpha_n u_n:
\alpha_i \in [0,1]\}
&\to 
\{\alpha_1 u_1 + \dots
+ \alpha_n u_n:
\alpha_i \in [0,1]\}\\
\sum_{k=1}^n x_k u_k
&\mapsto 
\sum_{k \neq i} x_k u_k
+ (1-x_i) u_i.
\end{align*}
The quotient of $\mathbb T^n$
by this group is a $n-hyperrectangle$
with lengths
$a_1/2,\dots,a_n/2$.

\end{example}

\begin{example}[Curved n-ball]

Consider the $n$-dimensional
sphere
\[\mathbb S^n =
\{(x_1,\dots,x_{n+1})
\in \mathbb R^n:\,
|x_1|^2 + \dots + |x_{n+1}|^2 = 1\}.\]
The quotient of $\mathbb S^n$
by the map
\begin{align*}
\mathbb S^n
&\to 
\mathbb S^n\\
(x_1,\dots,x_n,x_{n+1})
&\mapsto 
(x_1,\dots,x_n,-x_{n+1})
\end{align*}
is a curved $n$-ball.
\end{example}

\section{Proof of the existence
for the Gaussian case}\label{ex:sec.4}
\setcounter{equation}0
 We consider the following path measure 
\begin{equation*}
Q=\int_{M^3}{R(\cdot|X_0=x, X_{1/2}=z, X_1=y)\, \pi(\mathrm dx \mathrm dy)\gamma_{1/4}(\mathrm dz)},
\end{equation*} 
where $\gamma_{\sigma^2}$ denotes the density of $\NN(0,\sigma^2 \id)$. This measure is 
the analogue of \eqref{ex.E1}. 
\begin{prop}[Constraints and entropy: Gaussian case]
\label{prop:QGaussian}
The measure $Q$ satisfies the endpoints and marginal constraints~${Q_{01}=\pi}$ and $Q_t=\NN(0,1/4 \id)$ for every~${t\in [0,1]}$. If $H(\pi|R_{01})<\infty$ then $H(Q|R)<\infty$.
\end{prop}
			
\begin{proof}
The steps and arguments of the proof are the same as in Section \ref{ex:sec.2}. Firstly, as in the proof of Proposition \ref{prop:Q}, the endpoint condition 
$Q_{01}=\pi$ is obviously satisfied. Then, for 
$t \in [0,1/2]$, we have 
\begin{equation*}
Q_t =\int_{M^2}{R_t(\cdot|X_0=x, X_{1/2}=z)\, \gamma_{1/4}(dx)\gamma_{1/4}(dz)}
\end{equation*}
where $R_t(\cdot|X_0=x, X_{1/2}=z)$ is the law, at time $t$ of a Brownian bridge on $[0,1/2]$ between~$x$ and $z$. It is a normal distribution $\NN\left((1-2t)x+2tz, t(1-2t)\right)$. So $Q_t$ is a normal distribution and we have that
\begin{equation*}
(1-2t)Y+2tZ+\sqrt{t(1-2t)}W \sim Q_t,
\end{equation*} 
where $Y$,$Z\sim\NN(0,1/4 \id)$ and $W\sim\NN(0,\id)$ are independent random variables. It follows that $Q_t=\NN(0,1/4\id)$ for all $t \in [0,1/2]$ and for all $t \in [0,1]$ with the same argument.
				
It remain to verify the entropy condition. As in the symmetric space case, we have
\begin{equation*}
H(Q|R) = H(\pi|R_{01}) + \int_{M}{H(\gamma_{1/4}|R^{xy}_{1/2})\, \pi(dxdy)}.
\end{equation*}
Using the heat kernel in $M$, we have :
\begin{equation*}
\frac{d R^{xy}_{1/2}}{d\gamma_{1/4}}(z)=e^{2\langle z,x+y\rangle-\frac{1}{2}|x-y|^2}.
\end{equation*}
And then, the entropy is 
\begin{equation*}
H(\gamma_{1/4}|R^{xy}_{1/2})= \frac{1}{2}|x-y|^2
\end{equation*}
So we have that
\begin{align*}
H(Q|R) 
&\leq H(\pi|R_{01})  +\int_{M^2}\frac{1}{2}|x-y|^2\,\pi(dxdy)\\
&\leq H(\pi|R_{01})  +\int_{M^2}(x^2+y^2)\,\pi(dxdy)\\
&\leq H(\pi|R_{01})  +2\int_{M}x^2\,\gamma_{1/4}(dx)\\
&\leq H(\pi|R_{01}) + \frac{n}{2}
\end{align*}
which completes the proof
\end{proof}
					

\begin{proof}[Proof of Theorem \ref{theo:Gaussian}]
It follows the proof of
Theorem \ref{theo:Homogeneous} but now using
Proposition \ref{prop:QGaussian} instead
of Proposition \ref{prop:Q}.
\end{proof}

 \section*{Acknowledgements}

This work
has been benefited
by conversations
with Marc Arnaudon
and Michel Bonnefont.
DGZ was supported by the French
  ANR-16-CE40-0024 SAMARA project
  and also would like to 
  thank the hospitality of Université de Bordeaux. 
  

\bibliographystyle{plain}
\bibliography{biblio2}

\noindent
\small{
Institut de Mathématiques de Marseille; CNRS; Aix-Marseille Université, Marseille, France.\\ 
URL : \url{https://davidgarciaz.wixsite.com/math}\\
\textit{E-mail address}: \texttt{david.garcia-zelada@univ-amu.fr}}

\bigskip

\noindent
\small{
Institut de Mathématiques de Bordeaux, UMR CNRS 5251, Université de Bordeaux, France\\
URL : \url{https://www.math.u-bordeaux.fr/~bhuguet/}\\
\textit{E-mail address}: \texttt{baptiste.huguet@math.u-bordeaux.fr}}	
\end{document}